\documentclass[11pt,a4paper, envcountsame,mathserif]{amsart}
\usepackage[usenames,dvipsnames]{color}

\usepackage[utf8]{inputenc}	
\usepackage{pdfpages,url}

 \usepackage[colorlinks,citecolor=blue,urlcolor=blue, linkcolor=blue, backref]{hyperref}

\usepackage[capitalise]{cleveref}		% Use \cref without environment name

 \usepackage{amsmath, amssymb, xspace, enumitem}
 \usepackage{graphicx}

 \usepackage[all]{xy}

\usepackage{tikz}
\usetikzlibrary{arrows,snakes,positioning,backgrounds,shadows}
\usepackage{stmaryrd}

\newtheorem{theorem}{Theorem}[section]

\newtheorem{thm}[theorem]{Theorem}

\newtheorem{fact}[theorem]{Fact}

\newtheorem{prop}[theorem]{Proposition}

\newtheorem{claim}[theorem]{Claim}

\newtheorem{notation}[theorem]{Notation}

\newtheorem{conjecture}[theorem]{Conjecture}

\newtheorem{lemma}[theorem]{Lemma}

\newtheorem{cor}[theorem]{Corollary}

\theoremstyle{definition}
\newtheorem{definition}[theorem]{Definition}
\newtheorem{example}[theorem]{Example}
\newtheorem{remark}[theorem]{Remark}

\DeclareMathOperator \SL{SL}
\DeclareMathOperator \GL{GL}
\DeclareMathOperator \PGL{PGL}

\newcommand{\Op}{\text{\it Mult}}
\newcommand{\Inv}{\text{\it Inv}}

\newcommand{\Tree}[1]{\mathit{Tree}(#1) }

\newcommand{\NN}{{\mathbb{N}}}
\newcommand{\RR}{{\mathbb{R}}}

\newcommand{\QQ}{{\mathbb{Q}}}
\newcommand{\ZZ}{{\mathbb{Z}}}
\newcommand{\sub}{\subseteq}
\newcommand{\sN}[1]{_{#1\in \NN}}
\newcommand{\sNp}[1]{_{#1\in \NN^+}}
\newcommand{\uhr}[1]{\! \upharpoonright_{#1}}

\newcommand{\bi}{\begin{itemize}}
\newcommand{\ei}{\end{itemize}}
\newcommand{\bc}{\begin{center}}
\newcommand{\ec}{\end{center}}

\newcommand{\ES}{\emptyset}

\newcommand{\tp}[1]{2^{#1}}

\newcommand{\ex}{\exists}
\newcommand{\fa}{\forall}

\newcommand{\la}{\langle}
\newcommand{\ra}{\rangle}

\newcommand{\n}{\noindent}

\newcommand{\vsp}{\vspace{6pt}}

\newcommand{\sss}{\sigma}
\newcommand{\aaa}{\alpha}

\renewcommand{\S}{S(\omega)}

\newcommand{\lland}{\, \land \, }

\newcommand \seq[1]{{\left\langle{#1}\right\rangle}}

\newcommand\+[1]{\mathcal{#1}}

\newcommand{\wt}{\widetilde}
\newcommand{\ol}{\overline}

\newcommand{\ape}{\, \hat{\ } \, }

\newcommand{\lra}{\leftrightarrow}
\newcommand{\LR}{\Leftrightarrow}
\newcommand{\RA}{\Rightarrow}
\newcommand{\LA}{\Leftarrow}

\newcommand{\rapf}{\n $\RA:$\ }
\newcommand{\lapf}{\n $\LA:$\ }

\newcommand{\sssl}{\ensuremath{|\sigma|}}

\newcommand{\range}{\ensuremath{\mathrm{range}}}
\newcommand{\dom}{\ensuremath{\mathrm{dom}}}

\DeclareMathOperator{\Aut}{Aut}

\numberwithin{equation}{section}
\renewcommand{\hat}{\widehat}
\begin{document}

\title{Logic Blog 2022}

 \author{Editor: Andr\'e Nies}

\maketitle

%\begin{abstract}  %The 2015  logic blog has focussed on the following:
% \end{abstract}

\setcounter{tocdepth}{1}
\tableofcontents

 {
The Logic   Blog is a shared platform for
\bi \item rapidly announcing  results and questions related to logic
\item putting up results and their proofs for further research
\item parking results for later use
\item getting feedback before submission to  a journal
\item fostering collaboration.   \ei

Each year's   blog is    posted on arXiv  2-3 months after the year has ended.
\vsp
\begin{tabbing}

 \href{https://arxiv.org/pdf/2202.13643.pdf}{Logic Blog 2021} \ \ \ \   \= (Link: \texttt{http://arxiv.org/abs/2202.13643})  \\

 \href{https://arxiv.org/pdf/2101.09508.pdf}{Logic Blog 2020} \ \ \ \   \= (Link: \texttt{http://arxiv.org/abs/2101.09508})  \\
 
 \href{http://arxiv.org/pdf/2003.03361.pdf}{Logic Blog 2019} \ \ \ \   \= (Link: \texttt{http://arxiv.org/abs/2003.03361})  \\

 \href{http://arxiv.org/pdf/1902.08725.pdf}{Logic Blog 2018} \ \ \ \   \= (Link: \texttt{http://arxiv.org/abs/1902.08725})  \\
 
 \href{http://arxiv.org/pdf/1804.05331.pdf}{Logic Blog 2017} \ \ \ \   \= (Link: \texttt{http://arxiv.org/abs/1804.05331})  \\
 
 \href{http://arxiv.org/pdf/1703.01573.pdf}{Logic Blog 2016} \ \ \ \   \= (Link: \texttt{http://arxiv.org/abs/1703.01573})  \\
 
  \href{http://arxiv.org/pdf/1602.04432.pdf}{Logic Blog 2015} \ \ \ \   \= (Link: \texttt{http://arxiv.org/abs/1602.04432})  \\
  
  \href{http://arxiv.org/pdf/1504.08163.pdf}{Logic Blog 2014} \ \ \ \   \= (Link: \texttt{http://arxiv.org/abs/1504.08163})  \\

   \href{http://arxiv.org/pdf/1403.5719.pdf}{Logic Blog 2013} \ \ \ \   \= (Link: \texttt{http://arxiv.org/abs/1403.5719})  \\

    \href{http://arxiv.org/pdf/1302.3686.pdf}{Logic Blog 2012}  \> (Link: \texttt{http://arxiv.org/abs/1302.3686})   \\

 \href{http://arxiv.org/pdf/1403.5721.pdf}{Logic Blog 2011}   \> (Link: \texttt{http://arxiv.org/abs/1403.5721})   \\

 \href{http://dx.doi.org/2292/9821}{Logic Blog 2010}   \> (Link: \texttt{http://dx.doi.org/2292/9821})  
     \end{tabbing}

\vsp

\n {\bf How does the Logic Blog work?}

\vsp

\n {\bf Writing and editing.}  The source files are in a shared dropbox.
 Ask Andr\'e (\email{andre@cs.auckland.ac.nz})  in order    to gain access.

\vsp

\n {\bf Citing.}  Postings can be cited.  An example of a citation is:

\vsp

\n  H.\ Towsner, \emph{Computability of Ergodic Convergence}. In  Andr\'e Nies (editor),  Logic Blog, 2012, Part 1, Section 1, available at
\url{http://arxiv.org/abs/1302.3686}.}

%
%\vsp 

%\n {\bf Announcements on the wordpress front end.}  The Logic Blog has a \href{http://logicblogfrontend.hoelzl.fr/
%}{front-end}  managed by Rupert H\"olzl.   
%
%\n (Link: \texttt{http://logicblogfrontend.hoelzl.fr/})

%\vsps
%When you post source code on the logic forum in the dropbox, you can post a comment on the front-end alerting the community, and possibly summarising the result in brief.  The front-end is also good for posting questions. It allows MathJax.
 
The logic blog,  once it is on  arXiv,  produces citations e.g.\ on Google Scholar.
%\n  A. Taveneaux, \emph{Randomness Zoo}, Logic Blog, Section 6, available at
%
%\texttt{http://dl.dropbox.com/u/370127/Blog/Blog2011.pdf}.

%  \part{Reverse mathematics}
% 
%   \part{Computability theory and randomness}
% \part{Set theory}
%           \input{sections2022/Lam_Nies}
%      
%  
  \part{Group theory    and its connections to  logic}
         \section{Gardam's refutation of the Higman/Kaplansky unit conjecture}
\newcommand{\supp}{\mathit{supp}}
\begin{definition} An element $u$ of a ring $R$ with $1$  is called a \emph{unit} if there exists $w\in R$  such that $uw=wu=1$.  The units of $R$ with the ring multiplication  form a group, which is denoted~$R^\times$. \end{definition}
Note that if $R$ has no $0$-divisors, then it suffices to require that $uw=1$; this implies that $u(wu-1)= 0$ and hence $wu=1$. 

Given a commutative ring $R$ with $1$ and a group $G$, recall that the group ring $R[G]$ consists of the formal sums $w=\sum_{g \in S} r_g g$, where $S \sub G$ is finite, and $r_g \in R-\{0\}$, with the obvious ring operations. The empty formal sum denotes  the $0$ of the group ring.  The   set $S$ is called the support of $w$, and we write $S= \supp(w)$.  The ring  $R$ embeds into $R[G]$ via $0_R \mapsto 0$, $r \mapsto r e$ for $r \neq 0$.   
Clearly $G $ embeds into  $R[G]^\times $ via $g \mapsto 1 g$.

The unit  conjecture stood for 81 years.
\begin{conjecture}[Higman 1940/Kaplansky 1970] Let $G$ be a torsion free group and $K$ a field. Then the group ring $K[G]$ has only trivial units, namely the ones  of the form $kg$, where $k \in K^\times$ and $g \in G$.  \end{conjecture}
The unit conjecture goes back to 
  Higman's  1940 PhD thesis at the University of Oxford. He noted that  it  holds for abelian groups, and proved it more generally for locally indicable groups (groups such that  each nontrivial f.g.\ subgroup has a quotient isomorphic to~$\ZZ$).
   
  Kaplansky in his 1956 problem list  on the occasion of a   conference on linear algebras held at \href{http://www.theramsheadinn.com}{Ram's Head Inn}  on Shelter Island, Long Island, conjectured   that $K[G]$ has no zero divisors. 
  In the revised 1970 version, he reviewed progress on this problem~\cite[Problem 6]{Kaplansky:70}, and also   added Higman's unit conjecture in the discussion thereafter, crediting it to a problem list written up after a 1968 conference in Kishinev, Moldavia.  
  
  The idempotent conjecture states that $K[G]$ has no  idempotents other than  $0,1$. It was formulated in the late  1940's for reduced group $C^*$-algebras, but did not appear in print until sometime later. 
  
\subsection{Implications of properties of a  torsion free group $G$}
The  property of a  torsion free group that it satisfies the unit conjecture for a field $K$  can be seen in the context of a hierarchy  of properties. 
In the  following list,  the first four  properties don't mention a field.
 Each property implies the next; trivially (4) also implies (3). \begin{enumerate}   

\item $G$ is \emph{left orderable}
\item $G$ is \emph{diffus}e (Bowditch  \cite{Bowditch:00}): if $C \sub G$ is a  nonempty finite set,  then there is $c \in C$, called an extremal element,  such that \bc  for each $g \in G-\{1\}$, one has $gc \not \in C$ or $g^{-1}c \not \in C$.   \ec
\item $G$ has the \emph{unique product property} (Kaplansky, Passman): if $A,B \sub G$ are nonempty finite sets,  then there is a pair $\la a, b \ra \in A \times B$ that is the only pair in $A \times B$   yielding  the product $ab$. 
\item $G$ has the \emph{two unique products property}: if  in addition to the above,  not both $A$ and $B$ are singletons, then there exist \emph{two } such pairs.
\item $K[G]$ has only the  trivial units. 
\item $K[G]$ has no zero  divisors.
\item $K[G]$ has no  idempotents other than  $0,1$. 

\end{enumerate}

 The implications can be seen as follows:

 \medskip

\n (1)$\to $(2): let $a = \max A$. If $ga < a$ then $a < g^{-1}a$.

\medskip 

\n (2)$\to $(3): 
Let $C = AB$. Let $c \in C$ be extremal. If $c= ab = rs$ where \bc $a,r \in A$ and $b,s \in B$ and $g=ra^{-1} \neq e$, \ec then  $gab = rb \in C$ and also $g^{-1}ab = as \in C$.
  
\medskip 

\n (3)$\lra $(4):  Due to Strojnowski~\cite{Strojnowski:80}. He also shows that  e in the definition of the unique product property one can assume that $A=B$.  The proof shows that if the property fails then $A$ can be made arbitrarily large.
\medskip 

\n (4)$\to $(5): Given two elements of $K[G]$ so that not both have support a singleton, the two unique products mean that the support of the product also is not a singleton.
\medskip

\n  (5)$\to $(6): Recall that a ring $R$  is called \emph{prime} if $u,v\in R - \{0\}$ implies that there is $r\in R$ such that $urv\neq 0$. Every  zero divisors free ring is prime because in that case   one can choose  $r=1$. 
It is known that $K[G]$ is prime (it suffices to assume here that  $G$ has no nontrivial finite normal subgroup); see Passman~\cite[Th.\ 2.10]{Passman:11}.

Suppose that $ab = 0$ in $K[G]$ where $a,b \neq 0$, and take $c$ such that $\aaa:=bca \neq 0$. Then $\aaa^2 =0$.   We have   $(1+\aaa)(1-\aaa)=(1-\aaa)(1+\aaa) = 1$,  so $1-\aaa$ is a  unit. If it is trivial, then $1-\aaa = kg$, so $\aaa =1-kg \in K[\la g \ra] \cong K[\ZZ]$ is a nonzero    element with square 0 in the ring of univariate  Laurent polynomials over $K$,  which is contradictory since that ring is a domain.
 
 \medskip

\n (6)$\to $(7)  Any nontrivial idempotent $x$ is a zero divisor because $x(x-1)= 0$.  
  % unless $\chi(K)=2$, when one has to work a little harder.  

 \medskip 
 The following  nonimplications are known at present: 
 
 \n (2) $\not \to$(1) by  a counterexample in the appendix to~\cite{kionke.raimbault:16}. This is the fundamental group of a certain   closed orientable hyperbolic 3-manifold.  
 
 \medskip 
 
 \n   (6) $\not \to$(5) for the 2-element field, because $\mathbb F_2[P]$ is known to be a domain, but has nontrivial units as discussed below. The same works for other    positive characteristics by the work of Alan Murray providing nontrivial units in such characteristics,  \url{arxiv.org/abs/2106.02147}.
 
 \subsection{Can one express these properties in first-order logic?}
 We consider the question which of the foregoing  properties can be expressed by a computable set $\{\phi_n\colon  n \in \NN\}$ of first-order sentences in the language of groups. For the properties involving a field $K$, we mean the assertion to hold for all fields.
 
 (1) it remains to be settled whether this is expressible. This might well be someone's open question.
 
 (2)-(4) are expressible, using one sentence for each size $n$ of a finite set,   respectively  sizes of pairs of finite sets. To represent  a   finite subset  $C$  of $G$ of size $n$ using f.o.\ logic,  we use    variables $x_1, \ldots, x_n$  and  think of $C$ as being $\{ x_1, \ldots, x_n\}   $. So (2) is expressed by the collection of sentences 
 \[\phi_n = \forall x_1, \ldots, x_n \bigvee_{i\le n} \forall y [ \bigwedge_k yx_i \neq x_k \lor  \bigwedge_k y^{-1}x_i \neq x_k].\]
 They are easily turned into universal sentences.
 
 For (5)-(7) note that each field is embeddable into an algebraically closed fields. So  to express the versions of  (5)-(7) where the  quantification is over  all fields, we might as well assume the fields are  algebraically closed. We can also   fix a  characteristic $d \in \NN-\{ 0\}$, because then we can throw the sentences for all $d$ together.

 We do (6); (5) and (7) will be similar.  
For $r,s, \ell \in \NN^+$ and $\ell \le rs$, an onto  function 
\bc $T \colon \{ 1, \ldots,  r\} \times \{ 1, \ldots, s\} \to \{ 1, \ldots, \ell\} $  \ec is codes    a partial  multiplication table  of  a group $G$. The sentence   in the language of groups $\phi_T = \forall x_1, \ldots, x_r \fa y_1 ,\ldots, y_s \fa z_1 ,\ldots , z_\ell $
\[  \bigvee_{i< j\le r }x_i = x_j \lor  \bigvee_{i< j\le s }y_i = y_j \lor \bigvee_{i\le r\land k\le s} z_{T(i,k)} \neq x_iy_k \]
says that this table cannot be realized by column of pairwise distinct elements  and row of pairwise distinct elements.

For each  such $T$ we have a system of equalities and inequalities, 
\[E_T\equiv \bigvee_i \aaa_i\neq 0 \lland \bigvee_k \beta_k \neq 0  \lland \bigwedge_{u \le \ell} [ \sum_{T(i,k)=u} \aaa_i \beta_k =0],\]
where $\aaa_i, \beta_k$ range over a field~$K$.

 Given a finite function  $T$, we put  $\phi_T$ on  the list of sentences whenever this system has a solution in an algebraically closed  field of characteristic~$d$. Then the whole list says that   (6) holds for such a field. Since $ACF_d$ is decidable uniformly in $d$ (and the sentences get longer as the parameters for $T$ increase), the resulting set of all  sentences is computable. 
 
Let $(\phi_n)$ be the list for   (3), and $(\psi_m) $ be the list for (5). By the compactness theorem,  if  (5) implies (3), then each sentence $\phi_n$ is implied by a finite list   $\psi_r$, $r \le L(n)$ where $L$ is computable. So to  show that (5) does not imply (3), it would suffice to find an instance  (given by the size of a  set $A$) of the unique product property (3) such that no,  however large,  instance   of the unit conjecture (5) implies it.
 
 \subsection{Torsion free groups without the unique product property}
  Rips and Segev~\cite{Rips.Segev:87}  used small cancellation theory to provide the first example of a torsion-free group $G $ that fails to have the unique product property. 
  
  Define a torsion free group by  
  \bc  $P = \la a,b \mid a^{-1}b^2 a = b^{-2}, b^{-1}a^2 b= a^{-2}\ra$. \ec 
 This group $P$ was  first proposed by Passman in connection with the Kaplansky conjectures, and had appeared in various related investigations. 
Promislow~\cite{Promislow:88}   showed that there is a     14-element subset $A$   of $P$ such that $P$ fails to have the unique product property via  $A,A$. 

 $P$   is    known as the Passman 4-group, the  Hantzsche-Wendt group and also the Promislov group due to the authors of these   papers.  Note that $P$  is  an extension of a free abelian group of rank 3, $N=\la a^2, b^2, (ab)^2\ra$ by the Klein 4 group $P/N= \la Na, Nb\ra$.  So  $P$ is   a  crystallographic group; this means that the group is a cocompact discrete subgroup of the topological group $O(n) \ltimes \RR^n$, for some $n$. The latter is the group of isometries of $\RR^n$ with the Euclidean norm. 
 
 Let  $x= a^2, y=b^2$. One can write $P$ as an amalgam 
\[ \la x,b \mid x^b= x^{-1} \ra *_U \la y, a\mid y^a= y^{-1} \ra \] where $U = \la x,b^2\ra \cong \la a^2, y \ra$. This shows that $P$ is torsion free.

We note that the group can in fact be identified among the  crystallographic groups in dimension 3 as the unique torsion-free example with finite abelianization.  In dimension 2, the only torsion-free ones are free abelian and Klein bottle group, both locally indicable. It follows that all the other dimension 3 torsion-free space groups are locally indicable, and  so satisfy the conjecture.

Some more complex variants of $P$ also fail to have the unique product property. 
Carter~\cite{Carter:14} defined for $k>0$ the torsion free group 
\[ P_k  = \la a,b \mid a^{-1}b^{2^k} a = b^{-{2^k}}, b^{-1}a^2 b= a^{-2}\ra.\]
Thus $P= P_1$.  Generalising the case of    $P_1$, each group  $P_k$ is an amalgam of two Klein 4 groups along a copy of $\ZZ^2$,   letting $y= b^{2^k}$. This  implies that $P_k$ satisfies the 0-divisors conjecture by a result of Lewin~\cite{Lewin:72}.  Carter shows that $P_k$ fails the unique product property for each $k$, that the $P_k$ are pairwise non-isomorphic, and that $P_k$ does not contain $P$ as a subgroup for $k>1$. Also each  $P_k$ for $k>1$  has  a subgroup of finite index of the form $\ZZ^2 \times F$ where $F$ is a free group of rank $> 1$. 
  
  Craig and Linnell~\cite{Craig.Linnell:22} conjecture that each uniform pro-p group satisfies the unique product property. They also consider the obvious variants of $P$ with more than 2 generators, and show that they are torsion free. However, such variants embed $P$.

 \subsection{A counterexample to the unit conjecture}
Gardam~\cite{Gardam:21} published a counterexample to the   unit conjecture for $P$ with the field $ \mathbb F_2$.  
Each element of $\mathbb F_2[P]$ has a unique normal form  \[p + qa + rb + sab,\] where the ``coefficients" $p,q,r,s $ are in $\mathbb F_2[N]$, $N= \la x,y,z\ra$. 

Gardam's example is given by $\aaa = p + qa + rb + sab$ where
 \begin{align*}
        p &= (1 + x) (1 + y) (1 + z^{-1}) \\
        q &= x^{-1} y^{-1}  + x + y^{-1} z + z \\
        r &= 1 + x + y^{-1} z + xyz \\
        s &= 1 + (x + x^{-1} + y + y^{-1}) z^{-1}.
  \end{align*}
  
  The inverse is given by $\beta= p' + q'a + r'b + s'ab$ where \bc $p'= x^{-1} p^a, q'= x^{-1}a, r'= y^{-1}r, s'= z^{-1}s^a$.  \ec 
  
  Murray  (\url{arxiv.org/abs/2106.02147}, Thm.\ 3),  provides  counterexamples to the unit conjecture for $P$ in any positive characteristic. For characteristic 0 it remains unknown whether the unit conjecture holds. 
Suppose there is a finite   set $S\sub P$ such that  for infinitely many primes p, there is a nontrivial unit in $F_p[P] $ with  support of the unit and its inverse contained in $S$. Let   $ K$  be an  ultraproduct of these  fields, which has characteristic 0. Then $K[P]$ has a nontrivial unit  with the same support $S$. However for the   examples of Murray,  the size of the support depends on the characteristic.

 \subsection{Conversion to a Boolean satisfiability problem} To find this, Gardam  converted the problem into a satisfiability problem and used a SAT solver to find a satisfying assignment, which  yields the unit and its inverse.  If we fix a finite superset $S$ of the supports of $\beta$ and $\aaa$, then for each word $w$ in $S$ we need only one bit, its  coefficient in $\mathbb F_2$, So we use  variables $\ell_w$ and $k_w$ to say whether  the corresponding coefficients of $\beta$ and  of $\alpha$, respectively,  are~$0$ or~$1$. We can   write that $\aaa$ is nontrivial and that $\beta\alpha=1$ as a Boolean formula: 
  \bi \item Nontriviality  of $\aaa$ is expressed by saying that the disjunction of the $k_w$ for $w \in S - \{1\}$  is true.
 \item the parity of 
$ \sum_{g, g^{-1} \in S} (\ell_{g^{-1}} \land k_g)$ is $1$
  \item For each $w \in S^2 - \{1\}$ we have a condition  that 
  
  \n the parity of 
$ \sum_{ h,g \in S \lland  hg = w} (\ell_h \land k_g)$ is $0$.\ei

Using standard methods this can be   put  into 3-conjunctive normal form.
For the given example, one takes $S$ as the ball of radius 5 in the Cayley graph of $P$ (which has 147 elements). A  SAT solver can find the answer in a reasonable time (hours to days).  See the 2021 lecture \href{https://www.youtube.com/watch?v=is7Gw5SDPsQ}{Giles Gardam: Solving semidecidable problems in group theory}.

\subsection{The group of units is not finitely generated} Gardam shows that in a sense  $\mathbb F_2[P]^{\times}$ is much larger   than $T$;   it is not finitely generated, and contains non-abelian free subgroups.

 Gardam~\cite[Lemma 1]{Gardam:21}   provides a lemma with conditions that  imply  being a unit, and these can be checked for the unit in question on about 2 pages. This more general method allows to prove that the group of units is not finitely generated.   Gardam   shows  the existence of infinitely many non-trivial units using this   algebraic condition.  However, the idea why  it works is that  the group $P$  contains many isomorphic copies of itself. This appears to be  necessary for such a trick to work: the multiplication table of the unit and its inverse determine the subgroup they generate up to isomorphism (in this case). Also the group $P$  has a  finite outer automorphism group~\cite{Zimmermann:90}, so one  cannot prove that  the subgroup of the group of units generated by group-automorphic images of the  non-trivial unit is not finitely generated by looking at the abelianization of the group of units. The result  crucially needs the computation of the group of units for the infinite dihedral group $\la a, b \mid b^2=1, a^b = a^{-1}\ra $ (a quotient of $P$ via this presentation) by Mirowicz.

 \subsection{Is the    group of trivial units first-order definable in $K[G]$?} During discussions of Nies and Gardam in M\"unster,  the question emerged whether the group $T$ is  first-order definable  in the group of units of  $\mathbb F_2[P]$ (or, at least, in the ring $\mathbb F_2[P]$).

 As is well-known, if $ H < G$  are groups and  $\alpha \in K[H]$  is a unit  in $K[G]$ then it is a unit   in $K[H]$ (and conversely). We use this below without special mention.
 
  If  $T$ was  definable in $\mathbb F_2[P]^\times$ without parameters, then $T$ would be invariant under automorphisms of  $\mathbb F_2[P]^{\times}$. The following fact shows that under a quite general hypothesis on $G$,   the group $T$ is   not even closed under conjugation by a nontrivial unit. Since any automorphism of $G$ extends to an automorphism of $K[G]$, this shows that $T$ is also not definable in $K[G]$ when there are nontrivial units and  the hypothesis on $G$ is satisfied.

\begin{prop}[Gardam] \label{prop: noninvariant}Let $K$ be a field. Let $G$ be a torsion free group with no  infinite nonabelian  f.g.\ subgroup $V$ such that $V/Z(V)$ has finite exponent (for instance, this holds if $G$  is solvable).   Let $T=K^\times G $ denote the group of trivial units in $K[G]$. 
	
	 Then the  group   $T$ equals  its own normaliser in the group of   units of  $K[G]$.    \end{prop}
\begin{proof} Suppose for a contradiction that $u $ is a  nontrivial  unit in $K[G]$ such that  $T^u= T$. Let $S:=\supp(u)$. Let $a \in S$;  since  $a^{-1} u $ is also a nontrivial unit,   we   may suppose that $e \in S$. Let $n= |S| !$.    \begin{claim} Let $g \in G$.  Then $g^n $ centralises  the subgroup $V= \la S \ra$ of $T$.  \end{claim}
\n By hypothesis that $u$ normalises $T$, we have that $h:= ugu^{-1} \in T$. Via the augmentation map (the  ring homomorphism  $K[G] \to K$ adding up the coefficients),  we see that $h \in G$. The map $x \mapsto  h^{-1}xg$, $x \in K[G]$,  fixes $u$, and hence permutes its support $S $. So the map $x \mapsto  h^{-n}xg^n$ fixes $S$ pointwise. Since $e \in S$, $h^n= g^n$. So $g^{-n}wg^n = w$ for each $w \in V$, as claimed. It follows that the f.g.\ group $V/Z(V)$ has exponent $n$.  If $V$ is abelian then  the  group ring $K[V]$ has only trivial units because $V $ is orderable.  So $V$ is nonabelian, contrary to hypothesis on $G$.

If $G$ is solvable then  $V/Z(V)$ is solvable;  this implies it  is finite (by an easy induction on its  solvability length, and using that $V$ is finitely generated). By a result of Schur, this implies that $V'$ is finite, and hence trivial because $G$ is torsion-free.   So again $V$ is abelian. \end{proof}

%because the degree of Laurent polynomials over a domain is additive under multiplication
\begin{remark} It is   difficult to give examples of non-abelian  f.g.\  groups which are torsion-free and 
centre-by-finite exponent, which suggests that    Proposition~\ref{prop: noninvariant} holds for lots of  groups. Adjan~\cite{Adjan:71} provided  the first example, where the centre is cyclic  (and the quotient by the centre  is    a free Burnside group). \end{remark}
We next show that the   group $K^\times P$ of trivial units  is parameter definable in  $K[P]^\times$,  for any field $K$.  As before, we use the notation from Gardam~\cite{Gardam:21}. In particular $a,b$ are generators of $P$, and $x=a^2, y= b^2, z=(ab)^2$ are generators of the  group  $N\cong \ZZ^3$ that has index $4$ in $P$. We use $a,b$ as parameters. The defining formula is a disjunction of conjunctions of equations.
\begin{prop}[Gardam] Let $u \in K[P]^\times$. Then 
	$u \in K^\times P$ $\LR$ 
	
	\hfill gy at least one of $u, ua, ub$,  $uab$   commutes with $a^2$, $b^2$, and $(ab)^2$. \end{prop}
\begin{proof} \rapf Recall  that $P/N $ is the Klein 4 group generated by $aN $ and $bN$. So for each $u \in P$,  one of $u, ua, ub, uab$ is in $N$.  

\lapf We may assume w.l.o.g.\ that it is $ u$ that commutes with $a^2$, $b^2$ and $ (ab)^2$. Write it in the ``normal form" \bc $u=p + qa + rb + sab$ \ec where $p,q,r,s \in K[\ZZ^3]$.  Using that $ba^2 = a^{-2}b$,  one has
\begin{eqnarray*}  a^2 u & = &  pa^2 + qa^2 a+ ra^2b+ sa^2 ab \\  
 ua^2 &=&  pa^2 + qa^2 a+ ra^{-2}b+ sa^{-2}ab. \end{eqnarray*}

Since $u$ commutes with $a^2$,  all the coefficients are equal. So $r (a^2-a^{-2})=0$ and similar for $s$.  This implies $r=s=0$ since $K[\ZZ^3]$    is an integral domain. 
 Similarly, since $u$ commutes with  $  b^2$ we have $q = 0$.  Since the units in $K[\ZZ^3]$ are trivial, this implies that  $u \in K^\times P$. %(It looks like I cheated, because maybe it only becomes invertible after embedding into $K[P]$, but thinking coset-by-coset we see this doesn't help. That is, as is well-known, if $ H < G$  then $\alpha \in K[H]$  is a unit, or zero divisor if and only if it is a unit, or zero divisor respectively, in $K[G]$.)
\end{proof}

\subsection{A verification of Gardam's  counterexample  in GAP}

Gardam verified his counterexample in the computer algebra system GAP, and this was attached as an ancillary file to the first version  on the arXiv of~\cite{Gardam:21}. There is also  a formalized \href{siddhartha-gadgil.github.io/automating-mathematics/posts/formalizing-gardam-disproof-kaplansky-conjecture}{proof} by \href{http://math.iisc.ernet.in/~gadgil/}{Siddhartha Gadgil} and Anand Rao Tadipatri in the proof assistant system ``lean".
% available at \url{}.

The following GAP interaction was provided by Bettina Eick in Oct.\ 2022. Comments added by Andre Nies.

First we find the group  $P$  in the catalog.  It is a polycyclic group; such groups are given  by polycyclic   presentations (pcp).  
\begin{verbatim}
	gap> gg := List([1..219], x -> SpaceGroupPcpGroup(3,x));;
	gap> gg := Filtered(gg, IsTorsionFree);                 
	[ Pcp-group with orders [ 0, 0, 0 ], Pcp-group with orders [ 2, 0, 0, 0 ], 
	Pcp-group with orders [ 2, 0, 0, 0 ], Pcp-group with orders [ 2, 0, 0, 0 ], 
	Pcp-group with orders [ 2, 2, 0, 0, 0 ], 
	Pcp-group with orders [ 2, 2, 0, 0, 0 ], 
	Pcp-group with orders [ 2, 2, 0, 0, 0 ], 
	Pcp-group with orders [ 2, 2, 0, 0, 0 ], 
	Pcp-group with orders [ 3, 0, 0, 0 ], 
	Pcp-group with orders [ 3, 2, 0, 0, 0 ] ]
	gap> gg := Filtered(gg, x -> IdGroup(x/FittingSubgroup(x))=[4,2]);   
	%code for Klein 4 group
	[ Pcp-group with orders [ 2, 2, 0, 0, 0 ],    
	Pcp-group with orders [ 2, 2, 0, 0, 0 ], 
	Pcp-group with orders [ 2, 2, 0, 0, 0 ] ]
	gap> List(gg, AbelianInvariants);
	[ [ 4, 4 ], [ 0, 2, 2 ], [ 0, 4 ] ]. %C_4 x C_4,  
	gap> P := gg[1];
	Pcp-group with orders [ 2, 2, 0, 0, 0 ]
\end{verbatim}

Now we build the group ring $R$.

\begin{verbatim}
	gap> R := GroupRing(GF(2), P);                                        
	<algebra-with-one over GF(2), with 10 generators>. %
	gap> GeneratorsOfRing(R);
	[ (Z(2)^0)*id, (Z(2)^0)*g1*g5, (Z(2)^0)*g1, (Z(2)^0)*g2*g4, (Z(2)^0)*g2,    (Z(2)^0)*g3, (Z(2)^0)*g4, (Z(2)^0)*g5, (Z(2)^0)*g5^-1, (Z(2)^0)*g4^-1, 
	(Z(2)^0)*g3^-1, (Z(2)^0)*id ]
	% Z(2)^0 is the 1 of the ring
	
	gap> P;
	Pcp-group with orders [ 2, 2, 0, 0, 0 ]
	gap> g := Igs(P);   %induced  generating system
	[ g1, g2, g3, g4, g5 ]
	gap> a := g[1];
	g1
	gap> b := g[2];
	g2
	gap> a^2;
	g5^-1
	gap> b^2;
	g4^-1
	gap> (a*b)^2;
	g3^-1
	
\end{verbatim}

Next identify   group elements $a,b,x,y,z$ with their images in  $R$, where $x= a^2, y=b^2, z= (ab)^2$, and hence the normal abelian subgroup  $N$ of $P$  is freely generated by $x,y,z$. Note that \begin{verbatim} (Z(2)^0) \end{verbatim} is the internal notation for the multiplicative identity of $\mathbb F_2$.

\begin{verbatim}
	
	gap> a := a*One(R);  
	(Z(2)^0)*g1
	gap> b := b*One(R);
	(Z(2)^0)*g2
	gap> x := a^2;
	(Z(2)^0)*g5^-1
	gap> y := b^2;
	(Z(2)^0)*g4^-1
	gap> z := (a*b)^2;
	(Z(2)^0)*g3^-1
	
\end{verbatim}

One  writes the unit $w$ in this normal form by declaring these coefficients.
\begin{verbatim}
	gap> p := (One(R)+x)*(One(R)+y)*(One(R)+z^-1);
	(Z(2)^0)*g3+(Z(2)^0)*g3*g5^-1+(Z(2)^0)*g3*g4^-1+(Z(2)^0)*g3*g4^-1*g5^-1+(Z(2)^
	0)*id+(Z(2)^0)*g5^-1+(Z(2)^0)*g4^-1+(Z(2)^0)*g4^-1*g5^-1
	gap> q := x^-1*y^-1+ x+ y^-1*z + z;
	(Z(2)^0)*g4*g5+(Z(2)^0)*g5^-1+(Z(2)^0)*g3^-1*g4+(Z(2)^0)*g3^-1
	gap> r := One(R) + x + y^-1*z+ x*y*z;
	(Z(2)^0)*id+(Z(2)^0)*g5^-1+(Z(2)^0)*g3^-1*g4+(Z(2)^0)*g3^-1*g4^-1*g5^-1
	gap> s := One(R) + (x+x^-1+y+y^-1)*z^-1;
	(Z(2)^0)*g3*g4+(Z(2)^0)*g3*g5+(Z(2)^0)*g3*g5^-1+(Z(2)^0)*g3*g4^-1+(Z(2)^0)*id
	gap> w := p + q*a + r*b + s*a*b;
	(Z(2)^0)*g1*g2*g3*g4+(Z(2)^0)*g1*g2*g3*g5+(Z(2)^0)*g1*g2*g3*g5^-1+(Z(2)^
	0)*g1*g2*g3*g4^-1+(Z(2)^0)*g1*g2+(Z(2)^0)*g1*g3+(Z(2)^0)*g1*g3*g4^-1+(Z(2)^
	0)*g1*g5^-1+(Z(2)^0)*g1*g4^-1*g5+(Z(2)^0)*g2*g3*g4+(Z(2)^0)*g2*g3*g4^-1*g5+(
	Z(2)^0)*g2*g5+(Z(2)^0)*g2+(Z(2)^0)*g3+(Z(2)^0)*g3*g5^-1+(Z(2)^0)*g3*g4^-1+(
	Z(2)^0)*g3*g4^-1*g5^-1+(Z(2)^0)*id+(Z(2)^0)*g5^-1+(Z(2)^0)*g4^-1+(Z(2)^0)*g4^
	-1*g5^-1
\end{verbatim} 
The expression for the  inverse  $v$ of $w$, given below, is closely related to the expression for $w$. We merely need to modify the coefficients.  For instance, the effect of conjugating $p$ by $a$ is flipping  the signs of the exponents of $y$ and $z$. 
\begin{verbatim}
	gap> v := x^-1 * a^-1*p*a + x^-1*q*a + y^-1*r*b+ z^-1*a^-1*s*a*a*b;
	(Z(2)^0)*g1*g2*g3+(Z(2)^0)*g1*g2*g4+(Z(2)^0)*g1*g2*g5+(Z(2)^0)*g1*g2*g5^-1+(
	Z(2)^0)*g1*g2*g4^-1+(Z(2)^0)*g1*g3*g5+(Z(2)^0)*g1*g3*g4^-1*g5+(Z(2)^0)*g1+(
	Z(2)^0)*g1*g4^-1*g5^2+(Z(2)^0)*g2*g3*g4^2+(Z(2)^0)*g2*g3*g5+(Z(2)^
	0)*g2*g4*g5+(Z(2)^0)*g2*g4+(Z(2)^0)*g4*g5+(Z(2)^0)*g4+(Z(2)^0)*g5+(Z(2)^
	0)*id+(Z(2)^0)*g3^-1*g4*g5+(Z(2)^0)*g3^-1*g4+(Z(2)^0)*g3^-1*g5+(Z(2)^0)*g3^-1
\end{verbatim}

Finally, we are ready to try it out:

\begin{verbatim}
	gap> w*v;
	(Z(2)^0)*id
\end{verbatim}
This says that $wv = 1$, and so GAP has verified Gardam's result refuting  the 82-year old conjecture of Higman.

%\begin{verbatim}
%w= g1*g2*g3*g4+g1*g2*g3*g5+g1*g2*g3*g5^-1+g1*g2*g3*g4^-1+g1*g2+g1*g3+g1*g3*g4^-1+g1*g5^-1+g1*g4^-1*g5+g2*g3*g4+g2*g3*g4^-1*g5+g2*g5+g2+g3+g3*g5^-1+g3*g4^-1+g3*g4^-1*g5^-1+id+g5^-1+g4^-1+g4^
%-1*g5^-1
%\end{verbatim} 

         \section{Some undecidable properties of computable profinite, and discrete  groups}
We show, without any claim of originality,  that basic properties of computable abelian groups are undecidable, both in the profinite and in the discrete setting. For profinite groups, being pro-$p$ is not decidable, and being topologically generated by one element is also not decidable. (In fact, this is already shown by finite groups.) 

For infinite discrete groups, being finite is not decidable.

We begin by reviewing the relevant definitions.
The following originates in work of Mal'cev and Rabin (ind.) from the 1960s.
\begin{definition} \label{compStr}  A \emph{computable   structure}  is a structure such that the domain is a 
 computable  set   $D\sub  \NN$,  and the    functions and relations of the structure are computable.
%A  countable structure $S$ is called \emph{computably presentable} if some  computable structure $W$ is  isomorphic to  it. In this context we call $W$ a \emph{computable copy} of $S$. 
 \end{definition} 
 %%%
  \begin{example}  \label{ex:discrete}
 (a)   For each $k \ge 1$, the group  $GL_k(\mathbb Q)$  is computably presentable. To obtain   a  computable copy,  one fixes an algorithmic encoding of the rational $k \times k$ matrices by natural numbers, and lets the domain  $D$ be the computable set of numbers that encode a matrix with nonzero determinant. Since the encoding is algorithmic, the domain and the matrix operations are computable.

\n (b) It is not hard to verify that an $n$-generated group  $G$ has a  computable copy  $\LR$ if 
  its     word problem is decidable. For the implication  ``$\LA$", assume some effective bijection $F_n \leftrightarrow \NN$. Suppose $G= F_n /N$ for a computable normal subgroup $N$. Write $\sim_N$ for $uw^{-1} \in N$, and note that $\sim_N$ is a computable equivalence relation. Let $D$ be the set of least elements in the $\sim_N$-classes, which is a computable set. To show the binary group operation is computable on $D$, given $r, s \in D$, we let $f(r,s)$ be the element $t$ of $D$ such that $rs\sim_N t$. Similarly for the inverse.
%  
%   For  the implication from right to left, say,   fix an encoding of the elements of  the free group $F_n$ by natural numbers. We have   $G\cong F_n/N$ for a normal subgroup $N$ which is computable as a subset of $F_n$ by hypothesis. Let $D$ be the computable  set of least representatives in each coset of $N$. Given an element of $F_n$, one can compute its representative in $D$. Thus,  there is a computable copy of $G$ with domain~$D$.
  \end{example}

  \begin{definition}[Smith~\cite{Smith:81} and la Roche~\cite{LaRoche:81}]   A profinite group $G$ is computable    if $G= \varprojlim_t (A_t, \psi_t)$ for a computable diagram  $ (A_t, \psi_t)\sN t$ of  finite groups and epimorphisms~$\psi_t \colon A_t \to A_{t-1}$ ($t>0$).  \end{definition}
  
  Let $C_r$ denote the cyclic group of size $r$.   

\begin{prop} For profinite groups, \bi \item[(i)] being pro-$p$ is not decidable 
\item[(ii)]  being topologically generated by one element is   not decidable. \ei \end{prop}

\begin{proof} For each $n$ we build a group $G_n$ which has the property if $n$ is not in the halting problem $K$. We have $K = \bigcup_tK_t$ where $K_t$ is the set of elements enumerated by the end of stage $t$ (we may assume that  $K_1= \ES$).
   
Fix $n$. We build a diagram  of finite groups $(A_t)\sN t$  with natural maps $\phi_k\colon A_{t+1}\to A_t$  and let $G_n$ be   the inverse limit.   
%For compact, this is known to be  effectively equivalent to the definition of c.t.d.l.c. groups.

(i) Fix a prime $q\neq p$. Let $A_0 = C_p$. If $n \not \in K_{t+1} $ let $A_{t+1} = A_t$ and $\phi_t $ the identity. If $n \in K_{t+1} - K_t$ let $ A_t  = C_{pq}$ and $\phi_t $ the natural map $ C_{pq} \to C_p$.  If $n   \in K_{t} $ let $A_{t+1} = A_t$ and $\phi_t $ the identity.

(ii)  This is similar,  replacing  $C_{pq}$ by $C_p \times C_p$. Also, if   $n \in K_{t+1} - K_t$,   let $\phi_t$ be  the addition map  on $C_p$, otherwise, take the identity maps as before.
\end{proof}

\begin{prop} For infinite discrete groups, being finite is not decidable. \end{prop}
\begin{proof} We build a computable presentation of $G_n$, for convenience on a computable set $D \sub \ZZ$ (which is an interval of the form $[-n, n]$, or all of $\ZZ$). 

To build   $D$, at each stage $t$ we decide whether $t$ and $-t$ are elements.  Similar for the computable operations on elements that have been decided to be in $D$.

If $n \not \in K_{t+1}$ declare that $-t$ and $t$ are in $D$.  Declare addition and inverse of elements $a,b \in D$  such that $|a|, |b| \le t/2$ if not already declared, according to addition on $\ZZ$.

If $n \in K_{t+1} $  then declare that  $-t$ and $t$ are not in $D$.  If $t$ is least such, also declare addition of the remaining elements in $D$ so as to making a computable copy of $C_{2t-1}$ on the final $D= \{-t+1, \ldots, t-1\}$. 
%Being   compact: we start building a presentation of Z. At stage t when n goes into the halting problem we turn this into a presentation of 

\end{proof}

   \section{Duality between    locally Roelcke precompact  groups and meet groupoids }

Andr\'e gave two seminar talks at the Department of Mathematics at the University of M\"unster   in May/June.  They were centred on using countable   ``approximation  structures" to describe certain uncountable totally disconnected topological groups. The following notes form a polished version of  the material in the first of these  two talks.

 All topological groups will be countably based, and all isomorphisms between them  will be topological.  One says that a topological group $G$ is \emph{non-Archimedean} if  the neutral element has a neighbourhood basis consisting of open subgroups. It is well known \cite[Th.\ 1.5.1]{Becker.Kechris:96} that this property is equivalent to being isomorphic to a closed subgroup of $S_\infty$ (the group of permutations of $\NN$), which means  the automorphism group of a structure  with domain $\NN$.

 The first talk was on oligomorphic groups, the second on locally compact groups. In the former, the approximation structures are used to derive an upper bound on the Borel complexity of the isomorphism relation. The second uses the approximation structure to study what it means for a t.d.l.c. group to be computable.  This is recorded as an    extended version \cref{P2} below,  based on a talk series at a  Newcastle conference in October.

  The   following setting has  been discussed  in detail in Kechris et al.\   \cite{Kechris.Nies.etal:18}. The closed subset of a Polish space $X$ form a standard Borel space, called the Effros space. If $X$ underlies a Polish group, the closed subgroups form a Borel set in the Effros space (see e.g. \cite[Lemma 2.5]{Kechris.Nies.etal:18} for detail).   One can gauge the complexity of the isomorphism relation on subclasses of the closed subgroups of  $S_\infty$ via the usual Borel reducibility $\le_B$.  The following classes have been (at least to some degree) studied under this aspect.

\begin{definition} A closed subgroup $G$ of $S_\infty$ is called \begin{enumerate} \item oligomorphic: for each $n$ only finitely many $n$-orbits
\item quasi-oligomorphic: isomorphic to an oligomorphic group (but the canonical evaluation action   on $\NN$ may fail to be oligomorphic)
\item Roelcke pre-compact (R.p): for each open subgroup $U$, there is a finite subset $\aaa$ such that  $G= U \aaa U$; that is, $U$ has finitely many double cosets
\item locally compact (t.d.l.c.): the usual definition for the underlying topology; equivalently,  there is a compact open subgroup  (van Dantzig's theorem)
\item locally R.p: there is a R.p.\ open subgroup. \end{enumerate}\end{definition}

For background on (locally) Roelcke precompact groups see Rosendal~\cite{Rosendal:21} and also   Zielinski \cite{Zielinski:21}. This set of concepts works  in fact in a  setting much more general than closed subgroups of  $S_\infty$.  The Roelcke uniformity in a topological group $G$ is given by the basic  entourages $\{\la g,h\ra \colon  g \in VhV\}$ where $V$ ranges over identity neighbourhoods. A subset $A$ of a Polish group $G$ is Roelcke precompact if it is totally bounded in the Roelcke uniformity. That is, for each entourage $V$ there is a finite set $F $ such that  $A \sub \bigcup_{x \in F} B_V(x)$. And of course $G$ is locally R.p.\ if it has  a R.p.\ nbhd of the identity. Even for non-Archimedean groups,  our definition of locally R.p.\ might be a bit more restricted than the one given by Rosendal and Zielinski. The two  are equivalent if there is  a  version,  in the more general locally Roelcke precompact setting, of van Dantzig's theorem for t.d.l.c.\ groups. To my knowledge   no-one has obtained this so far.

\begin{figure}[h] \label{fig:diagram}

\[ \xymatrix { & \text{locally Roelcke precompact} & \\
\text{Roelcke precompact}\ar[ur] & & \text{locally compact}\ar[ul] \\
\text{quasi-oligomorphic}\ar[u] & \text{compact}\ar[ul]\ar[ur] & \text{discrete}\ar[u] \\
\text{oligomorphic}\ar[u] &&  }
\]

%
%\[ \xymatrix{    & { FUCK} & \\   %%%%%%% 1
%E_\infty\ar@{~>}[ur]^{<_B} & \cong_\text{Roelcke precompact}\ar@{~>}[u]^{\equiv_B}_{\text{\cite[Thm.\  3.1(ii)]{Kechris.Nies.etal:18}}} & \\    %%%%% 2
%\cong_\text{quasi-oligomorphic}\ar@{~>}[ur]^{\le_B}\ar@{~>}[u]_{\le_B}^{\text{present paper}}    &   & \cong_\text{compact}\ar@{~>}[ul]_{\equiv_B}   \\ %%%%%%%%
%   \cong_\text{oligomorphic}\ar@{<~>}[u]_{\equiv_B}^{\text{present paper}}   & &  {\GI}\ar@{~>}[u]^{\equiv_B}_{\text{\cite[Thm.\  4.3]{Kechris.Nies.etal:18}}}} \]

%\[ \xymatrix{    & { \GI}\ar@{~>}[dd]^{\equiv_B}_{\text{\cite[Thm.\  4.3]{Kechris.Nies.etal:18}}}} & \\   %%%%%%% 1
%E_\infty\ar@{~>}[ur]^{<_B} & \cong_\text{Roelcke precompact}\ar@{~>}[u]^{\equiv_B}_{\text{\cite[Thm.\  3.1(ii)]{Kechris.Nies.etal:18}}} & \\    %%%%% 2
%\cong_\text{quasi-oligomorphic}\ar@{~>}[ur]^{\le_B}\ar@{~>}[u]_{\le_B}^{\text{present paper}}    &   & \cong_\text{compact}\ar@{~>}[ul]_{\equiv_B}   \\ %%%%%%%%
%   \cong_\text{oligomorphic}\ar@{<~>}[u]_{\equiv_B}^{\text{present paper}}   & &  {\GI}\ar@{~>}[u]^{\equiv_B}_{\text{\cite[Thm.\  4.3]{Kechris.Nies.etal:18}}}} \]
%   
   
\caption{\small  Some classes  of closed subgroups of $S_\infty$, with their inclusion relations.}
\end{figure}
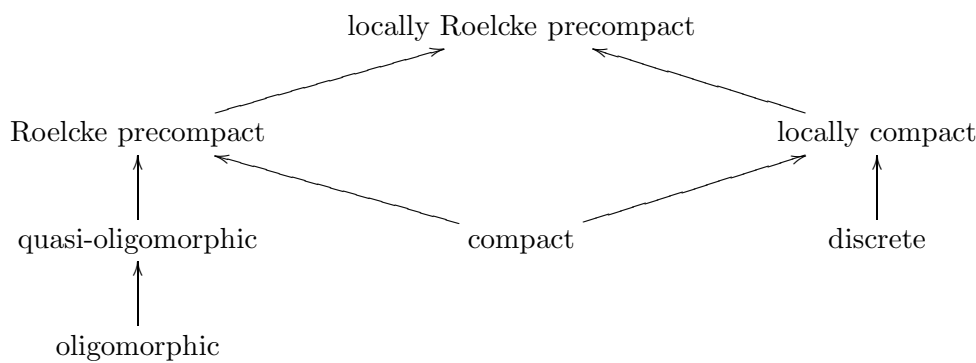

Since all classes are Borel, each inclusion in the diagram implies Borel reducibility between the isomorphism relations. Even for the locally R.p.\ case, an upper bound on the complexity of the isomorphism relation is graph isomorphism GI. This is shown by suitably extending \cite[Thm.\ 3.1]{Kechris.Nies.etal:18}. The idea is that $G$ acts faithfully from the left on a countable set of open cosets. For [quasi]-oligomorphic $G$, just take all open cosets; in general, take the R.p.\ open cosets. 

For all classes  $\+ C$ except [quasi-]oligomorphic, this bound is sharp by a profinite version (\cite{Kechris.Nies.etal:18}, last section) of the Mekler construction, which codes certain ``nice" countable graphs into countable nilpotent-2 groups of exponent $p$. For  [quasi-]oligomorphic, there is a tighter upper bound: the isomorphism relation is Borel reducible to a Borel equivalence relation with all classes countable (that is, below the orbit equivalence relation given by a Borel action of a countable group).

 \subsection{Meet groupoids of locally R.p.\ groups
 }  \label{ss: mg} Intuitively, the notion of a  {groupoid} generalizes the notion of a group by allowing that the binary operation is partial.  A groupoid   is  given by a domain $\+ W$ on which     a unary operation $(.)^{-1}$ and a partial binary operation, denoted  by ``$\cdot $", are defined. These operations satisfy the following conditions:
 \bi \item[(a)] associativity  $(A \cdot B)\cdot C= 
 A \cdot (B\cdot C)$,  with either both sides or no side defined (and so the parentheses can be omitted in products);  \item[(b)]  $A\cdot A^{-1}$ and $A^{-1}\cdot A$ are always defined; \item[(c)] if $A\cdot B$ is defined then $A\cdot B\cdot B^{-1}=A$ and $A^{-1}\cdot A\cdot  B =B$.\ei

It follows from (c) that a groupoid satisfies the left and right cancellation laws. One says that an element $U\in \+ W$ is \emph{idempotent} if $U\cdot U =U$. Clearly this implies that $U= U \cdot U^{-1}= U^{-1} \cdot U$ and so $U= U^{-1}$ by cancellation. Conversely, by (c) every element of the form $A\cdot A^{-1}$ or $A^{-1}\cdot A$ is idempotent.

  Note that if $A= Ua =aV$ and $B= Vb= bW$ for R.p.\ open subgroups $U,V,W$ then $AB= Uab= abW$.  Recall that a groupoid is equivalent to a category where each morphism has an inverse. In category notation, we write this now as  

\bc $U \stackrel A \to V \stackrel B \to W$  implies  $U \stackrel{AB} \to W$. \ec

Furthermore, if $C,D$ are left cosets of $U,V$, respectively, and $C\cap D \neq \ES$, then pick $a \in C \cap D$. We have  $C \cap D= a(U \cap V)$.  This motivates the next definition.

 \begin{definition} \label{def:MeetGroupoid} A \emph{meet groupoid} is a groupoid  $(\+ W, \cdot , {(.)}^{-1})$ that is also a meet semilattice  $(\+ W, \cap  ,\ES)$ of which  $\ES$ is the  least element.    
It satisfies the conditions  that  $\ES^{-1} = \ES = \ES \cdot \ES$,   that $\ES \cdot A$ and $A \cdot \ES$ are undefined for each $A \neq \ES$,  that $U  \cap V \neq \ES$ for idempotents $U,V$ such that $U,V \neq \ES$,
%\sasha{Does it have to be equal to $\ES$? I think maybe not in general.   What about $\ES^{-1}$? It should be safe to set it equal to $\emptyset$ in both cases.} \andre{we never need that, so don't need to define the product. The only purpose of $\ES$ is to get a meet semilattice.}
and that the groupoid operations are monotonic: writing $A \sub B \LR A\cap B = A$,
 
 \bi \item[(d)] $A \sub B \LR A^{-1} \sub B^{-1}$, and
 
\item[(e)]   if  $A_i\cdot B_i$ are defined ($i= 0,1$) and $A_0 \cap A_1 \neq \ES \neq B_0 \cap B_1$, then \bc $(A_0  \cap A_1)\cdot (B_0 \cap B_1) =  A_0 \cdot  B_0 \cap A_1 \cdot B_1 $ \ec

%\item[(e)] if  $A_i\cdot B_i$ are defined ($i= 0,1$) and $A_0 \sub A_1, B_0 \sub B_1 $,
%
%\n  then  $A_0 \cdot B_0 \sub A_1 \cdot B_1$. 
 \ei

Monotonicity of $\cdot$ follows from (e):  if  $A_i\cdot B_i$ are defined ($i= 0,1$) and $A_0 \sub A_1, B_0 \sub B_1 $,
 then  $A_0 \cdot B_0 \sub A_1 \cdot B_1$.     
 
  Note that (e) holds for the R.p.\  open cosets of a group: Let $A_i$ be a right coset of a subgroup $U_i$ and a left coset of subgroup $V_i$,  so that  $B_i$ is  a right coset of $V_i$ by hypothesis. Then  as noted  above, $A_0 \cap A_1 $ is a  right  coset of $U_0 \cap U_1$, and $B_0 \cap B_1$ a right coset of $V_0 \cap V_1$, so the left hand side is defined.   Clearly the left side is contained in the right hand side by monotonicity. The right hand side is also a right coset of $U_0 \cap U_1$, so they must be equal.

\begin{fact}  If $U$ and $V$ are idempotent,  then so is $S:=U \cap V$. \end{fact}
\begin{proof}   Just let $U= A_0= B_0$ and $V= A_1 = B_1$  in  (e). In fact this also follows without (e).    Since inversion is an order isomorphism, we have $S = S^{-1}$. In particular, $S \cdot S$ is defined. Since $S \sub U,V$, we have $S \cdot S \sub U\cdot U \cap V \cdot V$ by monotonicity  of $\cdot$, and hence $S \cdot S \sub S$. This implies $S \cdot S \cdot S^{-1} \sub S \cdot S^{-1}$ and hence $S \sub S \cdot S$.
\end{proof}  
 Given meet groupoids $\+ W_0, \+ W_1$, a bijection $h \colon \+ W_0 \to \+ W_1$ is an \emph{isomorphism} if it preserves the three operations. 
  Given a meet groupoid $\+ W$, the letters $A,B, C$ will  range over  elements of $\+ W$, and the letters  $U,V,W$ will range  over idempotents.  \end{definition}

\begin{definition} Let $G \le_c S_\infty$ be locally Roelcke precompact (in the sense that there is an open Roelcke precompact subgroup). Define a meet groupoid $\+W(G)$ as follows. The domain consists of all R.p.\ open cosets of $G$, together with $\ES$. The inverse $A \to A^{-1}$ is as usual. For $A, B \in \+ W(G)$,   let $A\cdot B$ be the usual product in case that $A $ is a left coset of some subgroup $V$, and $B$ is a right coset of $V$ (so that $AB$ is indeed a R.p.\ open coset). The operation $\cap $ is the usual intersection operation. 
\end{definition}

\begin{remark} Using the somewhat exotic notion of locally R.p.\ groups is necessary to define a common framework for the less exotic classes in Figure~\ref{fig:diagram}. For [quasi]-oligomorphic groups, every open coset is R.p., so one can ignore this restriction to R.p. For t.d.l.c.\ groups $G$, R.p.\ open cosets coincide with the  open compact ones. So $\+W(G)$ is simply the meet groupoid on compact open cosets. 
\end{remark}

\subsection{Borel duality} 
Recall that all   classes $\+ C$ in the diagram are Borel.  Using an old result of Lusin-Novikov one can ensure the meet groupoid has domain $\NN$ as described in the proof of \cite[Thm.\ 3.1]{Kechris.Nies.etal:18},  and  then view $\+ W$ as a Borel operator  from $\+ C$ to a suitable set of   structures with domain $\NN$.  For all the classes except oligomorphic, the  following map $\+ G$ is an inverse  of $\+ W$ up to isomorphism.  
Given a meet groupoid $M$ with domain $\NN$, let 

\medskip
$\+ G(M)=\{ p \in S_\infty \colon \, p \text{ preserves inclusion } \land $ 

\hfill $ p(A \cdot B) = p(A ) \cdot B \text{ whenever }   A \cdot B \text{ is defined } \}$. 

\medskip
 
So we have the duality

 a Borel class of groups $\xymatrix{{\mathcal C} \ar@/^/^{\+ W}[r]& {\mathcal D}\ar@/^/^{ \+ G}[l] }$   a Borel class of meet groupoids on $\NN$.

Duality    means that  for each $G \in \+ C$ and  $M \in \+ D$ \bc  $\+G(\+ W(G) ) \cong G$  and   $\+ W( \+ G(M)) \cong M$.   \ec
 $\+ D$  is  the closure under isomorphism of the range of $\+ W \uhr {\+ C}$, which  can be shown to be Borel. This is the domain of $\+ G $ which is also a Borel map.

 To show that  $\+G(\+ W(G) ) \cong G$  for each $G \in \+ C$ is not hard; essentially  this is done already  in \cite{Kechris.Nies.etal:18}.  That the closure  $\+ D$ of the range of $\+ W$ is Borel,  and showing the  dual condition  $\+ W( \+ G(M)) \cong M$, is much harder. Given a class $\+ C$, for membership of $M$  in $\+ D$ we need an extra Borel condition on $M$ that ensures  $\+ W( \+ G(M)) \cong M$. The problem is that there may be an open subgroup $\+ U$  in $\+ G(M)$ that isn't ``named" by an element of $M$. That is, $\+ U$ is not of the form $\{ p \colon \, p(U) = U\}$ for any $U \in M$. For an idempotent $U \in M$, let $L(U)= \{A \in M \colon AU= A\}$ be the set of left ``cosets". Similar define $R(U)$, the right cosets.    For the compact setting, say, one requires that $L(U) $ is always finite, and, if $L(U)= R(U)$ (i.e. $U$ is ``normal") and a set $\+ B \sub L(U)$ is closed under the groupoid operations, then there is an idempotent  $W \in M$ such that $A \in \+ B \lra A \sub W$  for each $A \in L(U)$. This is axiom CC in~\cite[Section~4]{LogicBlog:20}. Similar conditions appear to  determine the right $\+ D$ for the other classes; see  the axiom CLC for t.d.l.c. groups in~\cite[Section~4]{LogicBlog:20}, and~\cite{Nies.Schlicht.etal:21} for oligomorphic groups; note that for locally R.p.\ groups  this hasn't really been carried out in any detail. 

\subsection{Borel complexity of the isomorphism relation in the  oligomorphic case} Now let $\+ C$ be the class of oligomorphic groups.  
One can't use the  map $\+ G$ as above because the group $\+ G(M)$ is not necessarily oligomorphic even if $M = \+ W(G)$ for an oligomorphic group $G$. Instead, we define a Borel map $\wt {\+ G}$ which doesn't suffer from this flaw. One has to work harder, by only considering the action on one set $L(V)$ for a particularly nice idempotent $V$; for the detail see  \cite{Nies.Schlicht.etal:21}, which however works in the setting of ``coarse groups" which are countable structures with only ternary relation, and the map $\+  G$ yields   groups with domain the set of  certain filters on these coarse groups. (We speculate that if that  work was recast in the more recent  setting of meet groupoids, it would  shed  at least 5 pages.) $V$ is given as follows.

\begin{lemma} Each oligomorphic group $G$ has an open  subgroup  $V$ such that the natural left action $G \curvearrowright L(V)$ is oligomorphic and faithful.  Thus we get an homeomorphic embedding of $G$ into the symmetric group on $L(V)$.\end{lemma} 
\begin{proof}[Sketch of proof] Let $n_1, \ldots, n_k$ represent the 1-orbits. Let $V$ be the pointwise stabiliser of $\{n_1, \ldots, n_k\}$. \end{proof}

For membership of a meet groupoid $M$ in $\+ D$, one condition (axiom) requires the existence of such a $V$. Given $M$ one can pick such a $V$ in a Borel way. Let $\wt {\+ G}(M)$ be the closed subgroup of $S_\infty$ given by the action on $L(V)$.

We now discuss how to obtain the upper bound on the Borel complexity of $\cong_\+ C$, the isomorphism relation on $\+ C$.  An equivalence relation $E$ on a Borel space $X$ is called \emph{essentially countable} if $E \le_B F$ for some Borel eqrel $F$ with all classes countable (for example,  isomorphism on f.g.\ groups). 

\begin{thm}[\cite{Nies.Schlicht.etal:21}] \label{thm:NSThaha}  The isomorphism relation between oligomorphic groups is essentially countable. \end{thm}

\begin{proof} 
 By the above $ \cong _\+ C$ is $\equiv_B$ the isom.\ relation on an invariant  Borel class $\+ D$ of countable structures for a finite signature.
 We use a result of Hjorth and Kechris \cite{Hjorth.Kechris:95}: let $\sss$ be a finite signature. 
 
 Let $F$ be a countable fragment of the infinitary logic $L_{\omega_1, \omega}(\sss)$; this means that  $F$ contains f.o.\ logic, and is closed under things like substitution and taking subformulas.  Let $\+ D$ be an $\cong$-invariant Borel class of $\sss$-structures. Suppose for each $M \in \+ D$ there is a tuple $\ol a $ in $M$ such that $\mathrm{Th}_F(M, \ol a)$ (the theory in this fragment)  is $\aleph_0$-categorical. 
 
 Then (conclusion of the result) $\cong_\+ D$ is essentially countable.
 
 Let now $\sss$ be the signature of meet groupoids. By the Lopez-Escobar theorem, there is a sentence $\phi$ in $L_{\omega_1, \omega}(\sss)$ describing  $\+ D$. There also is a formula $\delta(x)$ describing what a ``faithful subgroup" is, i.e., a $V \in M$ as above. Let $F$ be the fragment generated by $\phi, \delta$. 
 \begin{claim}[\cite{Nies.Schlicht.etal:21}, Claim 4.5]  $\mathrm{Th}_F(M, \ol a)$   is $\aleph_0$-categorical. \end{claim}
 Now we can apply the Hjorth-Kechris result  to get the upper bound on the complexity of $\cong_{\+ C}$. 
 \end{proof}

\section{Ferov: A paradigm for computation with symmetries  of combinatorial structures, and visualisations thereof}
%\author{Michal Ferov\\Zero-dimensional symmetry research group}

\subsection{Candidate paradigm}

Visualisation of a combinatorial object and computation with its symmetries is a difficult task. The first problem is that the way in which an object is given can be straightforward to a mathematician with an experience in the field, but   does not have to translate in any obvious way to a computational or graphical software that is currently available.

We propose a modular approach, where  an abstract object, be that a combinatorial structure or an element of its automorphism group, is considered in three different ways and there is an algorithm (not necessarily uniformly given) that provides a transition from one way of representing the object to another. The paradigm can be summed up by the following diagram.

\begin{figure}[h]
    \begin{center}
    \begin{tikzpicture}[scale=0.15]
    \tikzstyle{every node}+=[inner sep=0pt]
    \draw [black] (38.9,-15.4) circle (4);
    \draw [black] (38.9,-15.4) circle (3.4);

    \draw (38.9,-15.4) node {User};
    \draw [black] (15.8,-32.4) circle (4);
    \draw (15.8,-32.4) node {Descr.};
    \draw [black] (38.9,-32.4) circle (4);
    \draw (38.9,-32.4) node {Data};
    \draw [black] (62.8,-32.4) circle (4);
    \draw (62.8,-32.4) node {Visual.};
    \draw [black] (36.7,-17.7) -- (18.25,-30.67);
    \fill [black] (18.25,-30.67) -- (19.2,-30.62) -- (18.62,-29.8);
    \draw [black] (18.8,-32.4) -- (35.9,-32.4);
    \fill [black] (35.9,-32.4) -- (35.1,-31.9) -- (35.1,-32.9);
    \draw [black] (41.9,-32.4) -- (59.8,-32.4);
    \fill [black] (59.8,-32.4) -- (59,-31.9) -- (59,-32.9);
    \draw [black] (60.36,-30.66) -- (41.34,-17.14);
    \fill [black] (41.34,-17.14) -- (41.71,-18.01) -- (42.29,-17.2);
    \end{tikzpicture}
    \end{center}
\end{figure}
The interpretation of the workflow within the diagram should be the following:
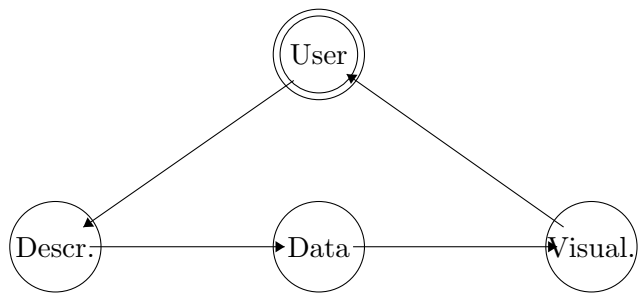
\begin{enumerate}
    \item the user provides a formal description of an object;
    \item based on the description given, an algorithm produces a data structure that captures the structure of the object, potentially with additional information;
    \item based on the given data, an algorithm produces a visualisation of the object.
\end{enumerate}

We explain how to interpret the terms used above. 
\begin{itemize}  \item \emph{Formal description}:  anything that can be passed to a computer. For example,  it could be a finite presentation of a group, a symmetric matrix describing a Coxeter system, a matrix describing a quasi-label regular tree, etc. This information should also always go together with a positive integer that specifies the size of the finite substructure of a possibly infinite object.

\item \emph{Data}:  content of  an appropriate data structure   representing the abstract object. In case of graphs,  the data would   capture the incidence structure of the graph.

\item \emph{Visualisation}:   any tool that, when provided with a an incidence structure of a graph, potentially with additional data such as labels, colouring, etc., will produce a graphical output that can be intuitively understood by the user. The user should have the option of tweaking the output based on their personal aesthetics and needs.
\end{itemize}
This modular approach has two  advantages. Firstly, the data can be stored and reused without having to recompute the incidence structure of an object every time the user wants to slightly modify the visual output, thus speeding up the process. Secondly,   this approach allows us to use tools that are readily available without having to reinvent the wheel for our own purposes.

\subsection{Graphs}
In the case of visualisation of graphs, the description will depend on the class of graphs we want to work with. For example, in the case of regular trees, the description might consist of a pair of integers $(d, R)$, where $d$ denotes the degree of the tree and $R$ denotes the radius, possibly with some  significant vertices.

The two most commonly used data structures to represent graphs  are   adjacency matrices and   neighbour (or adjacency)  lists. The meaning of an adjacency matrix is clear: is is a square matrix whose rows and columns are indexed by vertices of the graph. The entry on the position $(i,j)$ is an indicator whether or not the $i$-th and the $j$-th vertex of the matrix are connected by an edge. Adjacency matrix has many advantages, for example it allows direct use of methods from algebraic graph theory, but for large sparse graphs, i.e. graphs with low number of edges, it is very ineffective in terms of space. Indeed, the graph has $n$ vertices, then the matrix is of size $n^2$. In case of graphs that have uniform upper bound on the degree of a vertex, say $d$, we know that  the graph will have at most $\frac{d}{2}n$ edges. For large $n$ this could be extremely ineffective. For this reasons, we propose to use neighbourhood list, i.e. for each vertex we have a list of its neighbours. Since we are primarily interested in the study of graphs with uniformly bounded degree, we believe that list of neighbours is the appropriate data representation of the graph. For the sake of computability this does not really make difference, as both of these approaches are equivalent, but if effectiveness/complexity of the computation is to be considered, the difference will become significant.

A fairly standard format for representing the incidence structure of a graph, including labels and colours, is \href{https://en.wikipedia.org/wiki/DOT_(graph_description_language)}{\color{blue}DOT file}. The DOT file format is supported by specialised graph theoretic software such as
\begin{itemize}
    \item \href{https://www.gap-system.org/}{\color{blue}GAP} - a system for computational discrete algebra;
    \item \href{https://pallini.di.uniroma1.it/}{\color{blue}nauty and Traces} - programs for computing automorphism groups of graphs and digraphs;
    \item \href{https://networkx.github.io/}{\color{blue}NetworkX} - a Python package for the creation, manipulation, and study of the structure, dynamics, and functions of complex networks.
\end{itemize}
Clearly, this is an advantage as an output of any of the above tools can be used as an input for any of the above, thus potentially compensating for a lack of specific functionality within some package. For example, a Python script could be used build the  incidence structure of two graphs by using NetworkX. The  output can  then be  forwarded to nauty in order  to determine whether the  two graphs are isomorphic.

The DOT file format is also supported by most visualisation tools such as
\begin{itemize}
    \item \href{https://graphviz.org/}{\color{blue}GraphViz} - an open source graph visualisation software;
    \item \href{https://dot2tex.readthedocs.io/en/latest/}{\color{blue}dot2tex} - a tool designed to give graphs generated by Graphviz a more LaTeX friendly look and feel;
    \item \href{https://github.com/magjac/d3-graphviz}{\color{blue}d3-graphviz} - a JavaScript library   that renders DOT graphs and supports animated transitions between graphs and interactive graph manipulation in a web browser.
\end{itemize}

\subsection{Automorphisms of a tree}
In case of computing with automorphism of a regular tree, the user might present the automorphism in term of local actions. A $d$-regular tree $\mathcal{T}_d$ has a legal edge colouring using $d$-colours, i.e. a function $c \colon E\mathcal{T}_d \to \{1, \dots, d\}$ such that $c(e_1) \neq c(e_2)$ whenever the edges $e_1, e_2 \in E \mathcal{T}_d$ are adjacent. Clearly, each vertex $v \in V\mathcal{T}_d$ then ``sees'' each colour exactly once and an automorphism $\alpha \in \mathop{Aut}(\mathcal{T}_d)$ induces a permutation $\sigma(\alpha, v) \in \mathop{Sym}(d)$ given by
\begin{displaymath}
    \sigma(\alpha,v)(c(e)) = c(\alpha(e)),
\end{displaymath}
where $e \in E\mathcal{T}_d$ ranges over edges incident to $v$. If we pick a vertex $v_0 \in V\mathcal{T}_d$, then every vertex $w$ can be uniquely identified with a finite sequence $W \in \{1, \dots, d\}$. One can easily check that an each automorphism $\alpha$ is uniquely determined by a pair
\begin{displaymath}
    \left(w, (\sigma_W)_{W \in \{1, \dots, d\}^*}\right),
\end{displaymath}
where $w = \alpha(v_0)$ and $\sigma_W = \sigma(\alpha, W)$.

The description of an automorphism $\alpha \in \mathop{Aut}(\mathcal{T}_d)$ provided by the user can then be the pair $(w, \overline{\sigma})$, where $w = \alpha(v_0)$ and $\overline{\sigma}$ is some encoding of a function from $V\mathcal{T}_d$ to $\mathop{Sym}(d)$, possibly together with additional information, such as a positive integer $R$ denoting the radius around the vertex $v_0$.

The data would then consist of an appropriate data structure for storing the the local permutations for all the vertices within the distance $R$ from the vertex $v_0$. For example, in our \texttt{python} implementation we used the \texttt{dict} structure.

In this particular case, the visualisation might not necessarily be a graphical output, but an answer to a problem, such as the following.
\begin{enumerate}
    \item Does the automorphism fix a vertex? If yes, list all the fixpoints.
    \item Does the automorphism translate an edge? If yes, list its endpoints.
    \item Does the automorphism translate along an axis? If yes, compute the translation length and describe  all the vertices on the axis.
    \item Do two given automorphism ``agree'' on some vertex? If yes, list all such vertices.
    \item Are two given automorphisms conjugate? If yes, specify a conjugating element.   
\end{enumerate}

\part{Nies: Two talks on computably  t.d.l.c. groups}
\label{P2}

\section{Introduction} 

   These notes were used by Andre for two talks at the Newcastle conference ``Computable aspects of t.d.l.c. groups" which took place in October 2022. They   form a short   version of~\cite{Melnikov.Nies:22}.

The  first talk    introduced  two  notions of computable presentation of  a  t.d.l.c. group,  and show their equivalence.  The first notion relies on standard notions of computability in the uncountable setting. The second  notion restricts computation to a countable structure of approximations of the elements, the ``meet groupoid" of compact open cosets.  Based on this, I obtain various examples of computably t.d.l.c. groups, such as $\Aut(T_d)$ and some algebraic groups over the field of p-adic numbers.   The first talk also outlines the computability theoretic notions  that are needed.

The second talk      showed that   given a computable presentation of a  t.d.l.c.\ group,  the  modular function and the Cayley-Abels graphs (in the compactly generated case) are computable. We  discuss the open question whether     the scale function can be  non-computable. 
We  will give a criterion based on meet groupoids when    the computable presentation is unique up to computable isomorphism. 
We  explain why  the   class of computably t.d.l.c. groups  is closed under most of the constructions studied by Wesolek~\cite[Thm.\ 1.3]{Wesolek:15}.

We thank Stephan Tornier and George Willis for helpful conversations  on t.d.l.c.\ groups, and for providing references.

The talks  are centred on the following questions.
  
 \begin{enumerate}[label=(\alph*)]   \item How can one define a computable presentation of  a t.d.l.c.\ group?     
 
\n  Which  t.d.l.c.\ groups have such a presentation?

   \item  Given a computable  presentation of a t.d.l.c.\ group, are objects such as the rational valued Haar measures, the modular function, or the  scale function  computable?

%\item Which closure  {properties}  does the class of t.d.l.c.\ groups with a computable presentation have?
%
\item Do    constructions that lead from t.d.l.c.\ groups to new t.d.l.c.\ groups have algorithmic versions?
    \item When is      a computable presentation of a t.d.l.c.\ group unique up to computable isomorphism?

 \end{enumerate}

%

% .
  
\subsection{Background on t.d.l.c.\ groups}  \label{s:background}
  Van Dantzig~\cite{Dantzig:36}   showed that each t.d.l.c.\ group has a neighbourhood basis of the identity consisting of compact open subgroups.  
With Question~(a) in mind,   we discuss  six  well-known  examples of t.d.l.c.\ groups,  and indicate  a compact open subgroup when it is not obvious.  We will return to them  repeatedly during the course of   the   paper.
 
\begin{enumerate}[label=(\roman*)]   \item All countable discrete groups are t.d.l.c. 

\item All profinite   groups are t.d.l.c.
 
 \item $  ( \mathbb Q_p,+)$,  the additive group of   $p$-adic numbers   for a prime~$p$ is an example of a t.d.l.c.\  group that is  in neither of the two classes above.
   The additive group $\ZZ_p$ of $p$-adic integers forms a compact open subgroup. 

\item  The    semidirect product $\ZZ \ltimes  \QQ_p$  corresponding to the  automorphism  $x \mapsto px$ on $\QQ_p$, and   $\ZZ_p$  is a compact open subgroup.
\item Algebraic groups over local fields, such as $\SL_n(\QQ_p)$ for $n \ge 2$, are t.d.l.c. Here $\SL_n(\ZZ_p)$  is a compact open subgroup.
  
\item Given a connected countable undirected graph such that each vertex has finite degree, its   automorphism group is t.d.l.c. The stabiliser of any   vertex forms a compact open subgroup.  \end{enumerate} 

By convention,  all t.d.l.c.\ groups will be  infinite.
%By an  \emph{undirected  tree} we mean  a  connected graph without  cycles.   For $d \ge 3$, by   $T_d$   one denotes the undirected  tree  where each vertex has degree $d$. The   group $\mathrm{Aut}(T_d)$ was   first studied by Tits~\cite{Tits:70}. 
%He showed that   each proper open subgroup of $G$  is compact, and that  each compact  subgroup of $G$ is contained in   the stabilizer of a vertex, or in   the stabilizer of an edge.    

%A general reference for t.d.l.c.\ groups is Wesolek~\cite{Wesolek:18}. 

  \subsection{Computable structures: the countable case}   \label{s:compMath}
  Towards defining computable presentations, we first recall the    definition of a computable function on~$\NN$,  slightly adapted to our purposes in that we allow the domain to be any computable set. 
  \begin{definition} \label{def:computable} Given a   set $S \sub \NN^k$, where $k\ge 1$,  a  function $f \colon S \to \NN$ is called computable if there is a Turing machine that on inputs $n_1, \ldots, n_k $  decides whether the tuple of  inputs   $(n_1, \ldots, n_k) $ is in $S$, and if so outputs $f(n_1, \ldots, n_k)$.  \end{definition} 
One version of the  Church-Turing thesis    states that  computability in this sense   is the same as being computable by some algorithm.

A structure in the  model theoretic sense consists of  a nonempty      set $D$, called the domain,  with relations and functions defined on it.  The following definition was first formulated in the 1960s by Mal'cev  and Rabin  independently.
\begin{definition} \label{compStr}  A \emph{computable   structure}  is a structure such that the domain is a 
 computable  set   $D\sub  \NN$,  and the    functions and relations of the structure are computable.
A  countable structure $S$ is called \emph{computably presentable} if some  computable structure $W$ is  isomorphic to  it. In this context we call $W$ a \emph{computable copy} of $S$.  \end{definition} 
 %%%
%  \begin{example}  \label{ex:discrete}
%  For each $k \ge 1$, the group  $GL_k(\mathbb Q)$  is computably presentable. To obtain   a  computable copy,  one fixes an algorithmic encoding of the rational $k \times k$ matrices by natural numbers, and lets the domain  $D$ be the computable set of numbers that encode a matrix with nonzero determinant. Since the encoding is algorithmic, the domain and the matrix operations are computable. 

%\n (b) It is not hard to verify that an $n$-generated group  $G$ has a  computable copy  if and only if 
%  its     word problem is decidable. 
%  
%   For  the implication from right to left, say,   fix an encoding of the elements of  the free group $F_n$ by natural numbers. We have   $G\cong F_n/N$ for a normal subgroup $N$ which is computable as a subset of $F_n$ by hypothesis. Let $D$ be the computable  set of least representatives in each coset of $N$. Given an element of $F_n$, one can compute its representative in $D$. Thus,  there is a computable copy of $G$ with domain~$D$.
  %\end{example}
%\andre{already put references.putting refs without details is misleading, e.g. are there also other books besides PourElRich and Wei00?} \sasha{Aberth, Goodstein, Ko, Braverman-Yampolsky,...} 
\subsection{Computable structures:  the uncountable case}
 In the field of computable analysis (for   detail  see e.g.\ Pauly~\cite{Pauly:16} or Schr\"oder~\cite{Schroeder:21}),    to define computability for an  uncountable structure, one  begins by   representing all the elements  by ``names",  which are infinite objects simple enough to be    accessible to computation of oracle Turing machines.  Names usually are elements of  the set $[T]$ of  paths on some computable subtree $T$ of $\NN^*$ (the tree of strings with natural number entries).  For instance, a  standard name of a real number  $r$ is  a path coding a      sequence of rationals $\seq{q_n} \sN n$ such that $|q_n - q_{n+1}| \le \tp{-n}$ and $\lim_n q_n = r$.

 Via   Turing machines with  tapes that hold the input,  one can define computability of functions and relations on $[T]$. One    requires that the functions and relations of the uncountable structure are computable on the names.     This   defines computability on spaces relevant to computable analysis; for instance,  one can define that   a function on $\RR$ is computable.
  Since  each  totally disconnected Polish space is homeomorphic to $[T]$ for some subtree $T$ of $\NN^*$,   there is no need to distinguish between names and objects  in our setting.  
 An   \emph{ad hoc}  way to define computability  often  works  for  particular classes  of uncountable structures:     impose   algorithmic constraints on   the definition of the class. 
 
An example is  the definition of when 
  a profinite group $G$ is computable  due to  Smith~\cite{Smith:81} and la Roche~\cite{LaRoche:81}: $G= \varprojlim_i (A_i, \psi_i)$ for a computable diagram  $ (A_i, \psi_i)\sN i$ of  finite groups and epimorphisms~$\psi_i \colon A_i \to A_{i-1}$ ($i>0$).  
%Such  a definition often turns out to be equivalent to the general definition of computability   restricted to the class 
%For a further example, the definition 
%of a computable Banach space in \cite{PourElRich} was formulated specifically for separable normed linear spaces, but turns out to be   a special case of a more general approach \cite{MeMo,Pontr} that works for, e.g., Polish groups. IRRELEVANT TO THE PAPER, REMOVED
 
%  \andre{is this what we have in Proposition 7.3? If so we should mention it there} \sasha{yes}
%%%%%%%%%%%

We now    discuss the questions posed at   the beginning   in more detail.

\subsection{Computable presentations of t.d.l.c.\ groups}   \label{s: types of presentations}
We aim at a   robust definition of the class of   t.d.l.c.\ groups with  a computable presentation. We want this class to have  good algorithmic closure properties, and also ask that our definition extend the existing definitions for  discrete, and for profinite groups.  We  provide two  types of computable presentations, which  will  turn out to be equivalent: a t.d.l.c.\ group has a computable presentation of one type iff it has one of the other type.

 \emph{Computable Baire presentations.} One asks that the domain of $G$ is what we call an computably   locally compact subtree of $\NN^*$ (the tree of strings with natural number entries), and the operations are computable in the sense of oracle Turing machines.     
Baire presentations  appear to be the simplest and most elegant notion of computable presentation for general totally disconnected Polish groups.  However,     computable Baire presentations are hard to study because the domain is usually uncountable.

\emph{Computable presentations via a meet groupoid.} We introduce  an  algebraic structure  $\+ W(G)$ on the countable set of compact open cosets in $G$, together with $\ES$. This  structure is   a  partially ordered groupoid, with the usual set inclusion,  and multiplication of a left coset of a subgroup $U$ with a right coset of    $U$ (which is a coset).  The intersection of two compact open cosets is such a coset itself,  unless it is empty, so  we have  a meet semilattice.   A   computable presentation of    $G$ via meet groupoids is a computable copy of the meet groupoid of $G$ such that the index function on compact open subgroups, namely $U,V \mapsto | U \colon U \cap V|$,  is also computable. 

 \subsection{Which t.d.l.c.\ groups $G$ have computable presentations?} Dis\-crete groups,  as well as   profinite groups, have  a computable presentation as   t.d.l.c.\ groups if and only if they have one in the previously established sense from the 1960s and 1980s,   reviewed in Section~\ref{s:compMath} above.
  %Sasha: Moved this: For discrete groups we verify  this in Example~\ref{ex:discrete computable}, and for profinite groups in Example~\ref{ex:prof Baire} for the implication from  right to left, together with Prop.\ \ref{lem:proc} for the converse implication.
We provide  numerous examples of computable presentations for     t.d.l.c.\ groups  outside these two classes.   For $(\QQ_p, + )$  we    use  meet groupoid presentations. For   $\Aut(T_d)$  and   $\SL_n(\QQ_p)$   we use   Baire presentations.

It can be difficult to determine whether a particular t.d.l.c.\ group has a computable presentation.   Nonetheless, our  thesis  is that \emph{all} ``natural" groups that are considered in the field of t.d.l.c.\ groups have computable  presentations. An interesting testing ground for this thesis  is        given by  Neretin's groups~$\+ N_d$ of almost automorphisms of $T_d$, for $d \ge 3$; see for instance~\cite{Kapoudjian:99}.

%   \sasha{To do: talk more about closure properties and maybe about the abelian case a bit more. One of the key points was that the notion should be closed under natural operations etc. We have group extensions, these strange products, quotients, etc. Perhaps, countable direct limits (linear systems with injective maps) can also be added to this list.   }
% \andre{OK, will add that,  but the direct limits should be reserved for later paper, this one is already close to 40 pages}\sasha{true}

\subsection{Associated  computable objects}  \label{Assoc.Comp} Recall that    to a t.d.l.c.\ group $G$ we associate  its meet groupoid $\+ W(G)$, an algebraic structure on its compact open cosets. If $G$ is given by a computable Baire presentation, then we   construct a copy $\+ W  $ of the meet groupoid $\+ W(G)$ that is computable in a strong sense, essentially including the condition  that some  (and hence any) rational valued  Haar measure on $G$ is computable when restricted to  a function   $\+ W \to \RR$.
We will   show in \cref{cor:action} that   the left,  and hence also the right,  action of $G$ on $\+ W$ is  computable. 
We conclude  that the  modular function on $G$  is computable. If $G$ is compactly generated,  for each  Cayley-Abels graph  one can determine a  computable copy,  and any two     copies of this type  are computably quasi-isometric (\cref{prop:CA graph}).  Intuitively, this means that the large-scale structure of $G$ is a computable invariant.

Assertions that the scale function  is computable have been made   for particular t.d.l.c.\ groups in works such as Gl\"ockner~\cite{Glockner:98} and Willis~\cite[Section~6]{Willis:01}; see the   survey Willis~\cite{Willis:17}.
%, which also considers the scale of a general endomorphism of $G$.
 In these particular cases, it was generally clear what it  means that one can compute the scale $s(g)$:  provide an algorithm that  shows it.  One has to declare what  kind of   input    the algorithm takes; necessarily it has to  be some  approximation to $g$, as $g$   ranges over a potentially  uncountable domain. Our new framework allows us to give a precise meaning to the question whether  the scale function is computable for a particular computable presentation of  a     t.d.l.c.\ group, thus also allowing  for a precise negative answer.   This appears reminiscent of the answer to   Hilbert's  10th problem,  which asked for    an algorithm that decides whether a multivariate polynomial over $\ZZ$ has a zero.  Only after a precise notion of computable function was  introduced in the 1930s, it became possible to   assert   rigorously    that no such algorithm exists; the final negative answer was given in 1970 by Y.\ Matyasevich \cite{Matijasevic:70} (also see \cite{Matijasevic:93}). In joint work with Willis~\cite[Appendix 1]{Melnikov.Nies:22} we have shown that there is   a computable presentation of    a t.d.l.c.\ group $G$ such that  the scale   function  noncomputable for this presentation.  
One  can further ask whether  for some computably presented $G$,  the scale is non-computable for \emph{each} of its computable presentations. 
An even stronger negative result  would be that such a  $G$ can  be chosen to have a unique computable presentation  (see the discussion in \cref{ss:auto} below).

%The paper contains several examples  of computable presentations such that the scale function is computable, but in general dunno. May be presentation dependent. Looking ahead, for a strong negative answer we'd like $G$ with a unique computable presentation and the scale noncomputable.

%Algorithmic content of classical results and constructions
\subsection{Algorithmic versions of constructions that lead from t.d.l.c.\ groups to new t.d.l.c.\ groups}
Section~\ref{s:closure} 
  shows that the class of computably t.d.l.c.\ groups   is closed under suitable algorithmic versions of many constructions that have been studied in the theory of t.d.l.c.\ groups. In particular,    the constructions (1), (2), (3) and (6) described  in  Wesolek~\cite[Thm.\ 1.3]{Wesolek:15}    can be phrased algorithmically  in such a way  that they stay within the class of  computably t.d.l.c.\ groups; this provides   further evidence that our class is robust. These  constructions  are suitable versions, in our algorithmic topological setting, of 
  \bi \item passing to closed subgroups, \item taking group extensions via continuous actions,  \item forming ``local" direct products, and \item taking quotients by closed normal subgroups  \ei (see~\cite[Section~2]{Wesolek:15}  for   detail on these constructions). 
  The algorithmic  version of taking quotients (\cref{thm:closure normal}) is the most demanding; it uses  extra insights  from   the proofs that the various forms of computable presentation are equivalent. 
  
%  Several constructions lead to new examples of t.d.l.c.\ groups with computable presentations. E.g., for $n \ge 2$, after   defining  a computable presentation of  $\SL_{n+1}(\QQ_p)$ directly in \cref{prop: SL2}, we proceed to a computable presentation of $\GL_n(\QQ_p)$ via taking a closed subgroup, and then to $\PGL_n(\QQ_p)$ via taking a quotient.

  \subsection{When is a computable presentation  unique?} \label{ss:auto}
%Citing  Willis, \cite[Section~5]{Willis:17},  ``it is a truism that computation in a group depends on the description of the group".  
%In the present article, we apply our notion of computable presentability of a t.d.l.c.~group to give a formal version of  this statement. We also give some examples  where this general  statement   fails. 
  Viewing a   computable Baire presentation as a description, we are interested in the question whether such a description is unique, in the sense that between any two of them there is a computable isomorphism. 
Adapting terminology for countable structures going back to  Mal'cev,
   we will call such  a  group  \emph{autostable}.     If a t.d.l.c.\ group is  autostable,  then computation in the group can be seen as independent of its particular description.    
  \cref{thm:compCrit} reduces the problem of whether a t.d.l.c.\ group is autostable to the   countable   setting of meet groupoids.   
%  We apply it to show that   the additive group $\QQ_p$ of the $p$-adic    numbers is autostable, and so is $\ZZ \ltimes \QQ_p $. 
%  Proving the autostability of these groups requires more effort than the reader would perhaps expect. For other groups,  such as $\SL_n(\QQ_p)$ for $n \ge 2$ and $\Aut(T_d)$, we leave open whether a computable   presentation is unique up to  computable isomorphism. 

 \section{Computability on paths of rooted trees}
  \subsection{Computably locally compact subtrees of $\NN^*$}
%  For basics on computability theory see, e.g., the first two  chapters of~\cite{Soare:87}, or the first chapter of~\cite{Nies:book} which also contains notation on strings and trees. Our paper is mostly   consistent with the terminology of these two sources. They  also serve for basic concepts such as Turing programs, computable functions, as well as partial computable (or partial recursive) functions, which will be  needed from  Section~\ref{s:comp notions} onwards. In this section we  will   review   some    more specialized  concepts   related to computability.   

\begin{notation} {\rm Let $\NN^*$ denote the   set of strings with natural numbers as entries. We use letters $\sss, \tau, \rho$ etc.\ for elements of $\NN^*$.  The  set  $\NN^*$ can be seen as a directed tree: the empty string is the root, and the successor relation is given by appending a number at the end of a string.    We write $\sss \preceq \tau$ to denote that $\sss $ is an initial segment of $\tau$, and $\sss \prec \tau$ to denote that $\sss$ is a proper initial segment. We can also identify finite strings of length $n$+$1$ with partial functions $\NN \rightarrow \NN$ having finite support $\{0, \ldots, n\}$.
We then write $\tau_i$ instead of $\tau(i)$. By $\max (\tau)$ we denote  $\max \{\tau_i \colon \, i \le n\}$. 
Let  $h \colon \NN^* \to \NN$ be the canonical encoding given by $h(w)= \prod_{i< |w|} p_i^{w_i+1}$, where $p_i$ is the $i$-th prime number.}  \end{notation}

\begin{definition}[Strong indices for finite sets of strings]  For   a finite set $u \sub \NN^*$ let   $n_u=\sum_{\eta \in u} \tp {h(\eta) }$;  one says that $n_u$ is the  \emph{strong index} for~$u$.     \end{definition}
  We will usually  identify a finite subset of $  \NN^* $ with its strong index.  
 Unless otherwise mentioned, by a (directed) tree  we mean a nonempty subset~$T$ of $\NN^*$ such that $\sss \in T$ and $\rho \prec \sss$ implies $\rho \in T$.  
By  $[T]  $ one denotes  the set of  paths of a tree $T$.  Our trees usually have no leaves, so $[T]$ is a closed set in Baire space $\NN^\NN$ equipped  with the usual product topology. Note that    $[T]$ is compact if and only if  each level of $T$ is finite, in other words $T$ is finitely branching. 
For $\sss \in T$ let    \bc $[\sss]_T = \{X \in [T] \colon \sss \prec X \}$.  \ec That is,  $ [\sss]_T $ is the cone  of  paths on $T$ that extend  $\sss$.

%
%We are now ready for the main definition:  trees whose sets of paths are  the domains of computable copies of t.d.l.c.\ groups.
\begin{definition}[computably  locally compact  trees] \label{def:comploccompact} Let $T$  be a computable subtree of $ \NN^*$ without leaves such that only the root can have infinitely many successors. We say that  $T$ is \emph{computably  locally compact}, or c.l.c.\ for short, if  %%
    there is a computable     function $H \colon \NN \times \NN \to \NN $ such that,  if   $\rho \in T $ is a nonempty string, then    $\rho(i) \le H(\rho(0), i)$ for each $i < |\rho|$.      \end{definition}
    
    We note that in the paper version, the notion of c.l.c.\ trees  is somewhat broader: for some parameter $k \in \NN^+$, only strings on $T$ of length $< k$ can have infinitely many successors. This helps but is only of  technical importance. 
    
%    \begin{remark} The full description of $T$ contains  $H$ and a bound on the set $\{ n \in \NN \colon (n) \in T\}$, or otherwise  the information that this set is infinite. This matters  when one considers uniformity statements e.g.\ the equivalence theorem~\cref{thm:main2}, or sequences of uniformly c.l.c.\ trees, such as in \cref{prop: AAA}. The full definition is that of computably locally compact trees, Def.\ 2.4 in \cite{Melnikov.Nies:22}, where these data are also included. \end{remark} RUBBISH!!

%
%\sasha{(New) I think we better say that there are exactly $H(\sss, i)$ successors. I know that this is what you get from computability of $T$, but nonetheless.
%We can combine (2) and (3) by saying that the branching function $h: T \rightarrow \omega \cup \{\infty\}$ (measuring the number of immediate successors of each node) is computable. There is a related notion of a Haar locally computable graph that was apparently very popular in the 90ies; basically what we are saying is that $T$ is like that (as a graph); we should perhaps mention it too, maybe as a remark or a footnote.}\andre{1. If you make that change, the proofs below may become harder because there I compute the bounds, but not the exact number. Will check.  2. Yes please add that remark if you can find references. }

Given  such a  tree $T$, the compact open subsets of $[T]$ can be algorithmically encoded by natural numbers. The notation below     will be used throughout.
\begin{definition}[Code numbers for compact open sets] \label{defn: str index compact}  Let $T$  be a c.l.c.\ tree. For  a finite set $u\sub T - \{\ES\}$,  let \bc $\+ K_u= \bigcup_{\eta\in u} [\eta]_T$,   \ec (note that this set is compact). By a  \emph{code number} for a compact open set $\+ K\sub [T]$ we mean the strong index for a finite set $u$ of strings such that  $\+ K = \+ K_u$.  \end{definition} 

Such a code number  is not unique (unless $\+ K$ is empty).  So we will need to  distinguish between the actual compact open set, and any of its code numbers.  
 So  one  can   decide, given    $u\in \NN $ as an input,  whether $u$ is a  code number. Clearly, each  
    compact open subset of $[T]$ is of the   form $\+ K_u$ for some~$u$. 
    
     The following lemma shows that the basic set-theoretic relations and operations are decidable for sets of the form $\+ K_u$, similar to the case of finite subsets of $\NN$.
    \begin{lemma} \label{lem: comp index tree}  Let $T$ be  a c.l.c.\  tree.  Given  code numbers  $u, w$,
    
 \bi \item[(i)] one can compute  code numbers for $\+ K_u \cup \+ K_w$ and  $\+ K_u \cap \+ K_w$;
    
   \item[(ii)] one can decide whether $\+ K_u \sub \+ K_w$.  In particular, one can,  given  a  code number  $u\in \NN$,     compute the minimal  code number $u^*\in \NN$ such that $\+ K_{u^*}=\+ K_u$. 
   \ei \end{lemma}
  \begin{proof}  
  \n (i) The case of  union   is trivial. For the   intersection operation,  it suffices to consider the case that $u$ and $w$ are singletons. For strings $\aaa, \beta \in T$, one has  $[\aaa]_T \cap  [\beta]_T = \ES$ if $\aaa ,  \beta$ are incompatbile, and otherwise $[\aaa]_T \cap  [\beta]_T = [\gamma]_T$ where $\gamma $ is the longest common initial segment of $\aaa, \beta$. 
  
\n  (ii)   Let~$H$ be a computable binary function as in   \cref{def:comploccompact}.  It suffices to consider the case that  $u$ is a singleton. Suppose that  $\aaa \in T - \{\ES\}$.
 The algorithm to  decide whether  $[\aaa]_T  \sub \+ K_w$ is as follows. Let $N$ be the maximum length of a string in $w$. Answer ``yes" if       for each $\beta \succeq \aaa$ of length $N$ such that $\beta(k) \le H(\aaa, k)$  for each $k < N$, there is $\gamma \in w$ such that $\gamma \preceq \beta$. Otherwise, answer~``no".
\end{proof}
% \bb
%     Note that each compact  open subset $K$  of $[T]$ has the form  $K= \bigcup_{\sss \in V} [\sss]$
% for some finite set of strings $V$.
%  If $K= \bigcup_{\sss \in D_e} [\sss]$ we call $e$ a \emph{strong index} for $K$.  
%Note that if $[\sss]\cap T$ is compact  

\begin{definition} \label{def:E} Given a c.l.c.\  tree $T$, let $E_T$ denote the set of \emph{minimal}  code numbers for compact open subsets of $[T]$.    By the foregoing lemma, $E_T$ is decidable. \end{definition}

\subsection{Computable functions on the set of paths of computable trees} \label{s:comp notions}

%This section provides   preliminaries on  computability of  functions that are defined on the set of paths of computable trees. These preliminaries will be used in Section~\ref{s:Baire} to introduce    computable Baire presentations of  t.d.l.c.\ groups, as well as  in much   of the rest of the paper.
Most of the  content of  this subsection can either be seen as a  special   case of known results  in abstract computable topology, or can be   derived from such results.

Let  $T$ be a computable subtree of $\NN^*$ without leaves.    To define that   a function which takes arguments from   the potentially  uncountable domain   $[T]$  is computable,   one descends to the countable domain of strings on $T$, where   the usual computability notions   work.  The first definition, Def.\ \ref{def:computable 1} below,   will apply   when we show  in \cref{cor:delta} that  the modular function on a computable presentation of a  t.d.l.c.\ group  is computable.  %
\begin{definition}  \label{def:computable 1}  \begin{enumerate} \item A function $\Phi: [T] \times \NN \to \NN$ is computable if there is an   {oracle  Turing machine} as follows. 
Given $f \in [T]$ and $w \in \NN$, 
when it has the   list of the values $f(0), f(1), f(2), \ldots $ written on the oracle tape, with sufficiently many queries of the type ``what is $f(q)$?" it can determine the value  $\Phi(f, w)$.   

A function $\Psi: [T]   \to \NN$ is computable if the function $\Phi(f, n)  = \Psi(f)$ (which ignores the number input)  is computable in the sense above.

\item A function $\Phi\colon [T] \to [\NN^*] $ is computable if and only if the function  $\wt \Phi: [T] \times \NN \to \NN$ given by $\wt \Phi(g, n)= \Phi(g)(n)$ is computable. 

\item Similarly, one defines that  $\Phi \colon [T] \times [S] \to [\NN^*] $ is computable, using a TM with  two oracle tapes.
\end{enumerate}

 \end{definition}
   
  \begin{example} Let $T= \NN^*$. The function  $\Phi(f,n)= \sum_{i=0}^n f(i)$ is computable. The oracle TM with $f$ written on the oracle tape  queries the values of  $f(i)$ for $i=0, \ldots,n$ one by one and adds them.  \end{example}
  
  Note that in general our functions will only be defined on $[T]$, not  on all of $[\NN^*]$. Thus, the oracle TM   only needs to  return an answer if the oracle $f$ is in~$[T]$.

\begin{remark} If $T$ is c.l.c., then  to say that $\Psi: [T]   \to \NN$ is computable means that there is a computable prefix free set $S$ of nonempty strings on $T$ and a computable function $\psi\colon S \to \NN$  such that $[T]= \bigcup_{\sss \in S} [\sss]_T$ and  $\Psi(f) = \psi(\sss)$ where $\sss $ is the unique string in $S$  such that $\sss \prec f$. As discussed at the meeting, this definition has     time-bounded  versions (assuming that the string entries are written in their binary expansion). One requires e.g.\ that $S$ and $\psi$ are  in $P$. Given this one can study whether the scale is polytime computable for particular presentations of t.d.l.c.\ groups. \end{remark}
 
For proofs of the following see \cite[Section 6]{Melnikov.Nies:22}.
 \begin{lemma} \label{lem:lahm} Suppose that $K$ and $ S$ are computable trees without leaves. Suppose further  that there is a computable function $H$ such that $\sss(i) < H(i)$  for each $\sss \in K$ and $i < \sssl$.
 Let   $\Phi\colon [K]\to [S] $ be computable via an oracle TM $M$. 
\bi \item[(i)] There is a computable function $\gamma$ as follows: in a computation of $\Phi(f,n)$, $M$ only needs queries up to $\gamma(n)$.  %$\forall n \in \NN \forall \rho \in K [|\rho| = g(n) \to \ex \tau \prec \rho |P_\Phi(\tau)| > n]$.
%
% |such that for each each $n$,   each string $\rho \in K$ of length $N= g(n)$ has a prefix  $\tau$ such that $P_\Phi(\tau)$ has length greater than $n$. 
%  
A function $\Psi: [T]   \to \NN$ is computable if the function $\Phi(f, n)  = \Psi(f)$ (which ignores the number input)  is computable in the sense above.

 \item[(ii)] If $\Phi$ is a bijection then $\Phi^{-1}$ is computable, via a partial computable function that is obtained uniformly in $K,H,S$ and $P_\Phi$. \ei \end{lemma}
 
  Intuitively, the   function $g$ in (i)   computes the ``$\delta$" in the definition of uniform continuity from the ``$\epsilon$": if $\delta= 1/n$ we have $\epsilon = 1/g(n)$.    
\begin{lemma} \label{lem:lahm2} Let $T,S$ be c.l.c.\  trees. Suppose  a function $\Phi\colon [T]\to [S] $ is computable via a partial computable function $P_\Phi $. Given  code numbers $u, w$, one can decide whether $\Phi(\+ K^T_u) \sub \+ K^S_w$. 
\end{lemma} 
%\begin{proof}   Suppose that  $u$ is a strong index for the set of strings $\{\aaa_1, \ldots, \aaa_r \} \sub T$, and $w$ is a strong index for the set of strings $\{\beta_1, \ldots, \beta_s \} \sub S$.  Let $K$ be the subtree of $T$  consisting of the prefixes, or extensions,  of some $\aaa_i$. Clearly one can uniformly obtain   a computable bound  $h$ for this $K$ as in Lemma~\ref{lem:lahm}. Let $n= \max_i |\beta_i|$. Let $N$ be the length computed  from $n$ through  that Lemma. Then $\Phi(\+ K^T_u) \sub \+ K^S_w$ if and only if  for each $\aaa \in K$ of length $N$, there is an $i$ such that   $P_\Phi(\aaa) \succeq \beta_i$. By   Lemma~\ref{lem:lahm}, this condition is decidable. 
%\end{proof} 

To prove \cref{prop: SL2} below, we will need a criterion on whether,  given  a computable subtree   $S$  of a c.l.c.\ tree $T$ (where $S$ potentially has   leaves),  the maximally pruned subtree  of $S$  with the same set of paths is computable.
 \begin{prop}  \label{prop: prune} Let $T$ be a c.l.c.\ tree. Let $S $ be a computable subtree of $T$, and suppose that there is a uniformly computable dense sequence $(f_i)\sN i$ in $[S]$. Then the tree $\wt S= \{\sss  \colon [\sss]_S \neq \ES\}$ is   decidable. (It follows that $\wt S$ is c.l.c. Of course, $[\wt S] = [S]$.)  \end{prop}
\begin{proof}   Given a   string $\sss \in T$,  if $\sss = \ES$ then $\sss \in \wt S$. Assuming $\sss \neq \ES$, we can compute the least   $t\in \NN$ such that $\sss \prec f_t$, or $\rho \not \in S$ for each $\rho \in T$ of length $t$ such that $ \rho \succeq \sss$; the latter condition can be decided by the hypothesis that  $T$ is c.l.c. Clearly $\sss \in \wt S$ iff the former condition holds. 
\end{proof}

\section{Defining computably t.d.l.c.\  groups via  Baire presentations} \label{s:Baire} 
 Each  totally disconnected Polish space  $X$ is homeomorphic to $[T]$ for some tree $T \sub \NN^*$; see  \cite[I.7.8]{Kechris:95}. Clearly  $X$ is locally compact iff for each $f \in [T]$ there is an $n$ such that the tree above $f\uhr n$ is finitely branching;   we can then assume that only the root can be infinitely branching.  This suggests to work, in the algorithmic setting,  with a domain of the presentation that has   the form $[T]$ for a  computably locally compact  tree $T$, and   require that the group operations on $[T]$ be computable according to \cref{def:computable 1}. The same approach   would work for other types of algebraic structure defined on $[T]$ for a computably  locally compact tree $T$, e.g. computably t.d.l.c.\ rings.

%  Given a computable tree $T$ without leaves, recall from Def.\ \ref{defn: str index compact} that by $\+ K_u$ we denote the   open subset of $[T]$  with code number $u\in \NN$. That is, $u$ is a strong index for a set of strings $\{\aaa_1, \ldots, \aaa_r \} \in T$ such that $\+ K_u= \bigcup_{i\le r} [\aaa_i]_T$. 
%%%%%%%%%%%%
 \begin{definition}\label{def:main1}  
A \emph{computable Baire~presentation}  
    is a topological group  of  the form  $H= ([T], \Op, \Inv)$ such that
  \begin{enumerate} \item  $T$ is computably   locally compact as defined in~\ref{def:comploccompact};
  \item  $\Op\colon [T] \times [T] \to [T]$ and $\Inv \colon [T] \to [T]$ are computable.
    \end{enumerate}
    We say that a t.d.l.c.\ group $G$ is \emph{computably t.d.l.c.}  (via  a Baire presentation) if  $G \cong H$ for such a group $H$.    \end{definition}

\begin{example} \label{prop: SL2} Let $p$ be a prime, and let $n \ge 2$.   Let $\QQ_p$   denote the ring of $p$-adic numbers. (i)  The t.d.l.c.\ ring $\QQ_p$ has a computable Baire presentation. (ii) The   t.d.l.c.\ group $\SL_n(\QQ_p)$  has a computable  Baire presentation. 
\end{example}
\begin{proof}
 (i)
Let $ Q$ be the tree of strings $\sss\in \NN^*$ such that    all entries, except possibly the first, are   among $\{0, \ldots, p-1\}$,  and
  $r0 \not \preceq \sss$ for each $r>0$.    We think of a  string $r \ape \sss \in Q$ as denoting the rational      $p^{-r} n_\sss\in \ZZ[1/p]$, where  $n_\sss$ is  the number which has  $\sss$ as a   $p$-ary expansion, written in reverse order: \bc $n_\sss = \sum_{i< \sssl} p^i \sss(i)$. \ec   We allow the case that $\sss$ ends in $0$. The   condition that  $r0 \not \preceq \sss$   for each $r>0$ says    that   $p$ does not divide~$n_\sss$.   
% 
% 
%  Note that $\aaa$ can be written uniquely in the form $p^{-r}m$ where $r \in \NN, m \in \ZZ$, and $p$ does not divide $m$ in case that $r>0$. We denote such $\aaa$ by the string $r \sss$ where $\sss$ is a $p$-ary expansion of $n$, written in reverse order.
  
   For instance,  let $p=3$;  then \bc $(3,1,0,2)$ denotes the rational   $3^{-3}\cdot  (1+ 2 \cdot 9)= 19/27$.  \ec 

For  the addition operation, consider an oracle Turing machine with two oracle tapes starting with notations $r\sss$ and $s\tau$ of numbers $p^{-r}m$ and $p^{-s}n$. Say $r \le s$. Then $p^{-r}m + p^{-s}n = p^{-s} (p^{s-r} m + n)$. Clearly the machine can  output a string denoting $p^{-r}m + p^{-s}n$. To continue the example above,  if the machine sees tapes starting with $(3,1,0,2)$ and $(4,1,2,0,0)$, it will internally replace the first string by $(4,0,1,0,2)$, and then keep the leading 4 and carry out the   addition  modulo  $3^4$ of the numbers $57$ and $10$ with base 3 expansions $(0,1,0,2) $ and $(1,2,0,0)$ respectively, resulting in $(4, 1, 0, 1,2)$. (This corresponds to  $19/27 + 10/81 = 67/81$.)

A similar argument works for multiplication.  It is important that we allow improper expansions i.e. strings ending in zeros as in the example above, so that  the operation of the machines is monotonic.

%  
% 
%
%We now provide   oracle TM as above. For  a string of the form~$r\ape \sss $~let 
%
%\medskip 
%   $P(r\ape \sss)=( r-k) \ape \tau$,   
%   \medskip
%   
%   \n  where $k\in \NN $ is maximal such that  $k \le r$  and 
%$0^k\preceq \sss $, and $\tau $ is the rest, i.e.\ $0^k \tau =  \sss $. To show  the computability of the function $q \to -q$, let
%   \medskip
%   
%    $P_1(r\ape \sss)= r\ape \tau$ where $|\tau|=\sssl$ and $\sss + \tau=0 \mod  p^{\sssl}$. 
%    
%    \medskip
%    
% \n    To show that the addition operation  is computable, let 
%
%
%  $P_2(r\ape \sss, s\ape \tau) = \begin{cases}
%    P (s \ape (0^{s-r} \sss + \tau) &\text{if $r \le s$}\\
%   P (r \ape (\sss + 0^{r-s} \tau )&\text{otherwise.}
%\end{cases}$
%    \n 
%    
%    
%  \n    To show that the multiplication  operation  is computable, let 
%
%
%  $P_3(r\ape \sss, s\ape \tau) =  \begin{cases} P ((2s) \ape (0^{s-r} \sss \cdot \tau)) & \text{if }  r \le s\\  
%     P ((2r) \ape (\sss \cdot  0^{r-s}  \tau )) & \text{otherwise. } \end{cases}$
%  
%  \n $P_3$ is the correct partial recursive function because, say for $r\le s$, in $\ZZ[1/p]$  one has \bc $p^{-r} n_\sss p^{-s}n_\tau= p^{-2s}n_{0^{s-r}\sss}n_\tau$. \ec
%

(ii) We   now provide a computable  Baire presentation $([T], \Op, \Inv)$ of  $\SL_n(\QQ_p)$.    Let $T$ be the computable  tree that is an  $n^2$-fold ``power" of $Q$. More precisely,
  $T = \{ \sss \colon \forall i < n^2 \, [ \sss^i \in Q]\}$, where $\sss^i$ is the string of entries of $\sss$ in positions of the form $kn^2 + i$ for some  $k,i \in \NN$. Note that $T$ itself is not c.l.c.\ as nodes up to level $n^2-1$ are   infinitely branching. However, we can assume it is by skipping the levels $1, \ldots, n^2-1$. Clearly, $[T]$ can be naturally identified with the matrix algebra $M_n(\QQ_p)$. By the  computability of  the ring  operations on $\QQ_p$ as verified above, the matrix product is computable as a function $[T] \times [T]\to [T]$, and the   function $\det \colon [T] \to [Q]$ is computable.  
  
  Basic computability theory shows   that for any c.l.c.\ trees $T$ and $R$, any computable path $f$ of $R$, and any computable function $\Phi \colon [T]\to [R]$, there is a computable subtree $S $ of $T$ such that $[S]$ equals the pre-image $\Phi^{-1}(f)$.    Applying  this to the function $\det \colon [T] \to [Q]$ and  the path  $f = 01000\ldots $  that denotes $1 \in \QQ_p$, we obtain a computable subtree $S$ of $T$ such that $[S]$ can be identified with $SL_n(\QQ_p)$.   Note that $S$ could have dead ends. We fix this next:
  
  It is well-known  that $SL_n(\ZZ[1/p])$ is dense in $SL_n(\QQ_p)$. This is a special case of strong approximation for algebraic groups (see \cite[Ch.\ 7]{Rapinchuk.Platonov:93}), but can also be seen in an elementary way using Gaussian elimination. The paths on $S$ corresponding to matrices in $SL_n(\ZZ[1/p])$ are precisely the ones that are    $0$ from some point on. Clearly there is a computable listing $(f_i)$ of these  paths.   So by \cref{prop: prune} we can  replace $S$ by a c.l.c.\ tree $\wt S$ such that $[\wt S] = [S]$.

 To obtain a computable Baire presentation based on $\wt S$, note that matrix multiplication on $[\wt S]$ is computable  as the restriction of matrix multiplication on $[T]$. To define the matrix inversion operation $\Inv$, we use the fact that   the inverse of a matrix with determinant $1$  equals its  adjugate matrix; the latter can be obtained  by computing determinants on minors.
  \end{proof}

\section{Defining computably t.d.l.c.\  groups via meet groupoids}
  This section provides the detail for  the second type (Type M) of computable presentations of t.d.l.c.\ groups   described in Section~\ref{s: types of presentations}.
 \subsection{The meet groupoid of a t.d.l.c.\  group}  Intuitively, the notion of a  {groupoid} generalizes the notion of a group by allowing that the binary operation is partial.  A groupoid   is  given by a domain $\+ W$ on which     a unary operation $(.)^{-1}$ and a partial binary operation, denoted  by ``$\cdot $", are defined. These operations satisfy the following conditions:
 \bi \item[(a)] associativity in the sense that $(A \cdot B)\cdot C= 
 A \cdot (B\cdot C)$,  with either both sides or no side defined (and so the parentheses can be omitted);  \item[(b)]  $A\cdot A^{-1}$ and $A^{-1}\cdot A$ are always defined; \item[(c)] if $A\cdot B$ is defined then $A\cdot B\cdot B^{-1}=A$ and $A^{-1}\cdot A\cdot  B =B$.\ei

It follows from (c) that a groupoid satisfies the left and right cancellation laws. One says that an element $U\in \+ W$ is \emph{idempotent} if $U\cdot U =U$. Clearly this implies that $U= U \cdot U^{-1}= U^{-1} \cdot U$ and so $U= U^{-1}$ by cancellation. Conversely, by (c) every element of the form $A\cdot A^{-1}$ or $A^{-1}\cdot A$ is idempotent.
 \begin{definition} \label{def:MeetGroupoid} A \emph{meet groupoid} is a groupoid  $(\+ W, \cdot , {(.)}^{-1})$ that is also a meet semilattice  $(\+ W, \cap  ,\ES)$ of which  $\ES$ is the  least element.     Writing $A \sub B \LR A\cap B = A$ and letting the operation $\cdot$ have preference over $\cap$,
it satisfies the conditions   \bi \item[(d)]   $\ES^{-1} = \ES = \ES \cdot \ES$,    and  $\ES \cdot A$ and $A \cdot \ES$ are undefined for each $A \neq \ES$,  
\item[(e)]   if $U,V$ are idempotents such that $U,V \neq \ES$, then   $U  \cap V \neq \ES$,  \ei
%\sasha{Does it have to be equal to $\ES$? I think maybe not in general.   What about $\ES^{-1}$? It should be safe to set it equal to $\emptyset$ in both cases.} \andre{we never need that, so don't need to define the product. The only purpose of $\ES$ is to get a meet semilattice.}

 \bi \item[(f)] $A \sub B \LR A^{-1} \sub B^{-1}$, and

%\item[(g)] if  $A_i\cdot B_i$ are defined ($i= 0,1$) and $A_0 \sub A_1, B_0 \sub B_1 $, then  $A_0 \cdot B_0 \sub A_1 \cdot B_1$. 

\item[(g)]   if  $A_i\cdot B_i$ are defined ($i= 0,1$) and $A_0 \cap A_1 \neq \ES \neq B_0 \cap B_1$, then \bc $(A_0  \cap A_1)\cdot (B_0 \cap B_1) =  A_0 \cdot  B_0 \cap A_1 \cdot B_1 $. \ec   
 \ei
 From (g) it follows and that the groupoid operations are monotonic: if  $A_i\cdot B_i$ are defined ($i= 0,1$) and $A_0 \sub A_1, B_0 \sub B_1 $, then  $A_0 \cdot B_0 \sub A_1 \cdot B_1$.
 Also,  if $U$ and $V$ are idempotent,  then so is $U \cap V$ (this can also be verified  on the basis of  (a)-(f) alone).
%\begin{proof}   Since inversion is an order isomorphism, we have $S = S^{-1}$. In particular, $S \cdot S$ is defined. Since $S \sub U,V$, we have $S \cdot S \sub U\cdot U \cap V \cdot V$ by monotonicity, and hence $S \cdot S \sub S$. This implies $S \cdot S \cdot S^{-1} \sub S \cdot S^{-1}$ and hence $S \sub S \cdot S$.
%\end{proof}  

 For  meet groupoids $\+ W_0, \+ W_1$, a bijection $h \colon \+ W_0 \to \+ W_1$ is an \emph{isomorphism} if it preserves the three operations. 
  Given a meet groupoid $\+ W$, the letters $A,B, C$ will  range over  general elements of $\+ W$, and the letters  $U,V,W$ will range  over idempotents of $\+ W$.  \end{definition}

We use set theoretic notation for the meet semilattice because  for the motivating  examples  of meet groupoids the intersection symbol means the usual. Note that  the intersection of two cosets is empty, or again a coset. 
\begin{definition} Let $G$ be a t.d.l.c.\  group.  We define a  meet groupoid $\+ W(G)$.  Its  domain consists of  the compact open  cosets in $G$ (i.e., cosets of compact open subgroups of $G$), as well as the empty set.  
 We define $A\cdot B$ to be the usual product $AB$ in case that  $A= B= \ES$, or $A$ is a left  coset of a subgroup $V$ and $B$ is a right  coset of $V$; otherwise $A \cdot B$ is undefined.  \end{definition}

 \begin{fact}   \label{fact:WGGW} $\+ W(G)$ is a meet groupoid  with the groupoid operations $\cdot$ and   $A \to A^{-1}$,  and the usual intersection operation $\cap$. \end{fact}

We will use the usual group theoretic terminology  for elements   of  an abstract meet groupoid $\+ W$. If $U$ is an  idempotent of $\+ W$ we call $U$ a \emph{subgroup},   if $AU = A$ we call $A$ a \emph{left  coset} of $U$, and if $UB= B$ we call $B$ a \emph{right coset} of $U$. Based on the axioms, one can verify that if $U \sub V$ for subgroups $U,V$, then the map $A \mapsto A^{-1}$ induces  a bijection between  the left cosets and the right cosets of $U$ contained in~$V$. 
%Also if $U \cap V$ is a subgroup for any subgroups $U,V$: clearly 
 
 {\small    We note that  
$\+ W(G)$ satisfies the    axioms of inductive groupoids   defined in Lawson~\cite[page~109]{Lawson:98}.    See~\cite[Section~4]{LogicBlog:20} for more on an axiomatic approach to meet groupoids.  	\begin{remark}  It is well-known~\cite{Higgins:71} that one can view  groupoids as  small categories in which every morphism has an inverse.  The elements of the groupoid are the morphisms of the category. The idempotent morphisms   correspond to the  objects of the category.  One has $A\colon U \to V$ where $U= A \cdot A^{-1}$ and $V= A^{-1} \cdot A$.
Thus, in  $\+ W(G)$, 
  $A\colon U \to V$ means that     $A$ is  a right coset of $U$ and a  left coset of $V$.    \end{remark}

 The idea to study appropriate Polish groups via an algebraic structure on their open cosets is due to   Katrin Tent,  and first appeared in  \cite{Kechris.Nies.etal:18}.    This idea was further elaborated in a paper  by Nies, Schlicht and Tent on the complexity of the  isomorphism problem for oligomorphic groups~\cite{Nies.Schlicht.etal:21}. There,       approximation     structures are used that are given by   the       ternary relation  ``$AB \sub C$", where $A,B,C$ are certain open cosets.  They are called ``coarse groups". 
In the present work, it will be important that we  have explicit access to  the combination of the groupoid and the meet semilattice structures (which coarse groups don't provide).  Coarse groups are too ``coarse" an algebraic structure to analyse algorithmic aspects of t.d.l.c.\ groups.}

The meet groupoid  $\+ W(G)$ might turn out to be a useful tool for studying $G$, apart from  algorithmic considerations.  For a locally compact group $G$, the group $\Aut(G)$  becomes a Polish group via the Braconnier topology, given by the sub-basis of identity neighbourhoods of the form 
\bc $\mathfrak A(K,U) = \{ \aaa \in \Aut(G)  \colon \, \forall x \in K  [\aaa(x) \in Ux \lland \aaa^{-1}(x) \in Ux] \} $, \ec
where $K$ ranges over the compact subsets of $G$, and $U$ over the  identity neighbourhoods of $G$.  As noted in \cite[Appendix A]{Caprace.Monod:11}, $\Aut(G)$ with this topology is Polish (assuming that  $G$ is countably based). For a t.d.l.c.\ group $G$, the following  shows that      $\Aut(G)$   can be viewed as    the automorphism group of a countable structure.

\begin{prop} \label{prop: Braconnier} Let $G$ be a t.d.l.c.\ group. The group $\Aut(G)$ with the Braconnier topology is topologically isomorphic to $\Aut(\+ W(G))$, via the  map $\Gamma$  that sends  $\aaa \in \Aut(G)$ to its action on $\+ W(G)$, that is $B \mapsto \aaa(B)$.  \end{prop}
\begin{proof} It is clear that  $\Gamma$ is an injective  group homomorphism. 
	To show that  $\Gamma$ is continuous, consider an identity neighbourhood of $\Aut(\+ W(G))$, which we may assume to have  the form $\{\beta \colon \, \beta(A_i)= A_i, i = 1, \ldots,  n\}$ where $A_i \in \+ W(G)$.  Let $U = \bigcap_i A_i A_i^{-1}$ and $K= \bigcup A_i$. Then $\aaa \in \mathfrak A(K,U)$ implies $\aaa(A_i) = A_i$ for each $i$. 

To show $\Gamma$ is onto, we explicitly describe its inverse $\Delta$. Recall  our convention  that  the variable $U$  ranges over the compact open subgroups of $G$.  Note that  $\bigcap_U gU= \bigcap _U Ug = \{g\}$. Since $\beta$ is an automorphism and the $\beta(gU)$ are compact,  $L=\bigcap_U \beta(gU)$ is non-empty.  If $|L|> 1$ then there are  $C,D\in \+ W(G)$ such that $C\cap D = \ES$, and  $C \cap L \neq \ES \neq D\cap L$. By compactness this implies $\beta^{-1} (C) \cap \{ g\} \neq \ES \neq \beta^{-1} (D) \cap \{ g\}$, a contradiction. So   \bc $\Delta(\beta)(g) = h \LR h \in \bigcap_ U \beta(gU)$\ec
	defines a map $\Delta\colon \Aut(\+ W(G)) \to G^G$.   \begin{claim} Let $\ol \beta =\Delta(\beta)$. We have $\ol \beta \in \Aut(G)$. \end{claim}
	Clearly $\Delta(\beta^{-1})= \ol \beta^{-1}$, so $\ol \beta$ is a permutation of $G$. Ignoring for the moment the difference between group elements and the singletons containing them, we have \bc  $\ol \beta (g^{-1})= \bigcap_U \beta(g^{-1}U) = \bigcap_V \beta(Vg^{-1} )= (\bigcap_V \beta(gV))^{-1}= (\ol \beta(g))^{-1}$. \ec
	Also one can verify that $\beta$ preserves the binary group operation, using that \bc $ \bigcap_U gU \bigcap_V hV= \bigcap_V ghV = \bigcap_V (gV^{h^{-1}}hV)$.  \ec If $g \in V$ then $\ol \beta(g) \in \beta(V)$, so $\ol \beta$ is open. Arguing the same for $\beta^{-1}$, we conclude that $\ol \beta $ is a homeomorphism of $G$.  This shows the claim.

	\begin{claim} If  $\beta \in \Aut(\+ W(G))$ and $B \in \+ W(G)$, then   $\Gamma(\Delta(\beta))(B) = \beta(B)$.  \end{claim} 
	By the definitions of $\Gamma$ and $\Delta$, we have   
	\bc $h \in \Gamma(\Delta(\beta))(B) \LR \exists g \in B [ h \in \bigcap_U\beta(gU)]$. \ec 
	Letting $V = B^{-1}B$, we have $B = gU$, so    $\Gamma(\Delta(\beta))(B) \sub \beta(B)$. For the converse inclusion, suppose that $h \not \in \Gamma(\Delta(\beta))(B)$. Then for each $g \in B$ there is $U_g$ such that $h \not \in \beta(g U_g)$. Since $B$ is compact and $B \sub \bigcup_{g \in B} U_g$, there is a finite set $S \sub G$ such that  $B \sub \bigcup_{g \in S} gU_g$.  
	\end{proof}
%	Then   $\beta(B) \sub \bigcup_{g \in S} \beta(gU_g)$: Otherwise, since $ \bigcup_{g \in S} \beta(gU_g)$ is closed, there is a $D \in \+ W(G)$, $D \neq \ES$ such that $ D \sub \beta(B)$ but $D \cap \beta(gU_g) = \ES$ for each $g \in S$; applying $\beta^{-1}$ this yields a contradiction to the fact that $B \sub \bigcup_{g \in S} gU_g$. We conclude that $h \not \in \beta (B)$, as required. This shows the claim. 

\subsection{Second definition of computably    t.d.l.c.\  groups}

 \begin{definition}  \label{def:comp_meet_groupoid} A meet groupoid $\+ W$ is called \emph{Haar computable}~if 
\bi \item[(a)]  its  domain is a computable subset $D$ of $\NN$; 

\item[(b)] the   groupoid and meet operations    are computable in the sense of \cref{def:computable}; in particular, the relation $\{ \la x,y \ra\colon \, x,y \in S \lland x\cdot y \text{ is defined}\}$   is computable;
 
 \item[(c)]     the partial function with domain contained in  $D \times D$ sending  a pair of subgroups $U, V\in \+ W $ to $|U:U\cap V|$ is   computable.   \ei
    \end{definition}
  Here   $|U \colon U \cap V|$  is  defined abstractly as the number of left, or equivalently right, cosets of the nonzero idempotent  $U \cap V$ contained in $U$;  we require implicitly that this number is always finite.  Note that by~(b), the partial order induced by the meet semilattice structure of $\+ W$  is computable. Also, (b) implies that being a subgroup is decidable when viewed as a property of elements of the domain  $S$; this is used in~(c). The condition (c)   corresponds to the computable bound  $H$ required in     Definition~\ref{def:comploccompact}.         For ease of reading we will say that $n \in D$ \emph{denotes} a coset $A$, rather than saying that~$n$ ``is" a coset. 
   
%{\color{red} We note that for compact (not necessarily totally disconnected) groups, computability of Haar measure is equivalent to `effective compactness' of the group; see~\cite{ArnoHaar,EffedSurvey} for the details.}
% 
 \begin{definition}[Computably t.d.l.c.\  groups via meet groupoids] \label{Def2}Let $G$ be a   t.d.l.c.\  group. We say that $G$ is \emph{computably t.d.l.c.}\  via a meet groupoid  if    $\+ W(G)$ has a Haar computable       copy  $\+ W$.  In this context, we call $\+ W$ a computable presentation of $G$ (in the sense of meet groupoids).   \end{definition}

  %A meet groupoid $\+ W$ satisfying the conditions (a) and (b) will be called  \emph{computable}.  

%  \sasha{I'd totally drop 'Haar' and would stick with 'a,b,c = computable' throughout. This is because: (1) Occam (because we do not really need computable without being 'Haar' so), and (2) we would then have to talk why there is a computable  but not Haar computable example. Such an example can be extracted from my Pontryagin paper I think (a profinite example), but this discussion would make Occam unhappy.  }
% \andre{Hi Hickam's defeat, I don't like Occam if it's used that way... to oversimplify. But seriously, I want to keep computable meet groupoid possible for other contexts, e.g. some day we will study in more detail the computable oligomorphic groups... There, (c) does not exist. Also, if anything, Malcev corresponds to (a) and (b) but not (c), so I thought you cared about that. I've replace  ``Haar" by ``Haar". See   the remark above whether it now appeases you. The Arizona has sunk.}   
%
 \begin{remark} \label{rem:Haar computable}  In this setting, Condition (c) of \cref{def:comp_meet_groupoid} is equivalent to  saying that every Haar measure $\mu$  on $G$   that assigns a rational number to some compact open subgroup  (and hence  is rational-valued)  is computable on $\+ W$, in the sense that the function assigning   to a compact open coset $ A$  the rational $\mu(A)$ is computable. Consider left Haar measures, say. First suppose that (c) holds. Given $A$, compute the  subgroup $V$ such that $A= A\cdot V$, i.e., $A$ is a left coset of $V$.   Compute $W = U \cap V$. We have  $\mu(A) = \mu(V)= \mu(U) \cdot  |V:W|/|U:W|$. 

Conversely, if the Haar measure is computable on $\+ W$, then (c) holds because $|U\colon V|= \mu(U)/\mu(V)$.
\end{remark}

   For discrete groups, the condition (c) can be dropped, as the proof of the following shows.
 \begin{example} \label{ex:discrete computable} A discrete group  $G$ is computably t.d.l.c.\ via a meet groupoid 
 $\LR$ $G$ has a computable copy in the usual sense of  \cref{compStr}. \end{example}
%  \sasha{The terminology is inconsistent, you mean Haar. O.w. it is actually incorrect as stated. Can we somehow at least get rid of ``Haar'' throughout? I'm sure we do not use it without Haar anywhere. }
% \andre{see comment later (search for ``nonsense"). Strongly because the condition on computability of indices of subgroups is beyond Malcev/Rabin}
 \begin{proof} For the  implication $\LA$,  we may assume that  $G$ itself is computable; in particular, we may assume that its domain is a computable subset of  $\NN$.  Each compact coset in  $G$  is finite, and hence can be represented by a strong index for a finite set of natural numbers. Since the group operations are computable on the domain, this implies that   the meet groupoid of $G$  has a  computable copy. It is then trivially Haar computable. 
 
 For the implication $\RA$, let $\+ W$ be a Haar computable copy of $\+ W(G)$.  Since $G$ is discrete, $\+ W$ contains a least subgroup $U$.  The set of left cosets of $U$ is computable, and forms a group with the groupoid and inverse operations.  This yields the required computable copy of $G$. \end{proof}
 
% \sasha{Wait, you say in  $\RA$ it is Haar computable, but this is not what the lemma states, and you do not need it here, right? It is sufficiently curious that we only need computable here, but I'd still drop Haar throughout the paper and stick with a+b+c  }\andre{for discrete, (c) can be dropped because the compact sets are finite so we know their sizes (i.e. counting measure)}
% 
 
By $\mathbb{Q}_p$ we denote the additive group of the $p$-adics. By the usual definition of semidirect products (\cite[p.\ 27]{Robinson:82}),   $\ZZ \ltimes \QQ_p$ is the  group defined on the Cartesian product $\ZZ \times \QQ_p$ via the binary operation  $\la z_1, \aaa_1 \ra \cdot  \la z_2, \aaa_2\ra = \la z_1 + z_2, p^{z_2} \aaa_1  + \aaa_2\ra $. 
%\begin{eqnarray*}    \la z_1, \aaa_1 \ra \cdot  \la z_2, \aaa_2\ra &=& \la z_1 + z_2, p^{z_2} \aaa_1  + \aaa_2\ra \\
%\Inv( \la z, \aaa \ra)  & = & \la g^{-1}, (\Phi (g^{-1}, h)^{-1}\ra . \end{eqnarray*}
This turns $\ZZ \ltimes \QQ_p$ into   a topological group with  the product topology. 

%\sasha{recall that $(\QQ_p,+)$ and $(\mathbb F_p((t)), +)$ denote ...} \andre{textbook stuff}

\begin{example} \label{ex:Qp} For any prime $p$,   the  additive group  $\QQ_p$ and   the group $\ZZ \ltimes \QQ_p$  are  computably t.d.l.c.\ via a meet groupoid.     \end{example}

\begin{proof} 
%\sasha{What about finite extensions of the p-adics?  See pages 13 (bottom) -14 (top) of ``An introduction to
%totally disconnected locally compact groups'', by
%Ilaria Castellano. See also Example 2.3.1 in this paper for a bunch of references -- a reference would not hurt. The last time I saw the proof of this classification was  in  Pontryagin's book ``Topological groups''; I think he did it independently of van Dantzig\footnote{Van Dantzig, D. (1931). Studien ber topologische Algebra. Dissertation, Amsterdam 1931}, but who knows (and who cares?).}
  We begin with  the additive group $\QQ_p$. Note that  its    open proper  subgroups  are  of the form $U_r:= p^{r}\ZZ_p$ for some  $r\in \ZZ$. Let $C_{p^\infty}$ denote the Pr\"ufer group $\ZZ[1/p]/\ZZ$, where $\ZZ[1/p]= \{ z p^{-k} \colon \, z \in \ZZ \lland k \in \NN\}$. 
 For each $r$ there  is a canonical  epimorphism  $\pi_r\colon  \QQ_p \to C_{p^\infty}$ with kernel $U_r$: 
   if  $\aaa= \sum_{i=-n}^\infty s_ip^i$ where    $0 \le s_i < p$, $n\in \NN$, we have
 \bc $\pi_r(\aaa) = \ZZ + \sum_{i=-n}^{r-1} s_ip^{i-r}$; \ec 
  here an empty sum is interpreted as $0$. (Informally, $\pi_r(\aaa)$ is obtained by taking the  ``tail" of $\aaa$ from the position $r-1$ onwards to the last position, and shifting it in order  to represent     an element of $C_{p^\infty}$.)   So   each compact open coset in $\QQ_p$  can be uniquely written in  the form $D_{r,a}= \pi_r^{-1}(a)$ for some  $r \in \ZZ$ and $a \in C_{p^\infty}$. The domain $S\sub \NN$ of the  Haar computable copy $\+ W$ of $\+ W(\QQ_p)$ consists of natural numbers canonically encoding such pairs $\la r,a\ra$. They will be identified with the cosets they denote. %In particular,  $U_r= D_{r,0}$.

The groupoid operations are computable because we have $D_{r,a}^{-1}= D_{r,-a}$, and $D_{r,a} \cdot D_{s,b}= D_{r, a+b}$ if $r=s$, and undefined otherwise.  
 It is easy to check that   $D_{r,a} \sub D_{s,b}$ iff $r \ge s$ and $p^{r-s}a=b$. So the inclusion relation is decidable.  We have  $D_{r,a} \cap D_{s,b}= \ES$ unless one of the sets  is contained in the other, so the meet operation is computable.   
 Finally,   for $r\le s$, we have $|U_r: U_s|= p^{s-r}$ which is computable.

Next, let  $G = \ZZ \ltimes \QQ_p$; we   build a  Haar computable copy $\+ V$ of $ \+ W (G)$. We will extend the listing $(D_{r,a})_{r \in \ZZ, a \in C_{p^\infty}}$ of compact open cosets in $\QQ_p$ given above.  For each   compact open subgroup of $G$, the projection onto $\ZZ$ is compact open, and hence  the trivial group.  So the only compact open subgroups of $G$ are of the form~$U_r$.  Let $g\in G$ be the generator of $\ZZ$  such that $g^{-1} \aaa g = p \aaa$ for each $\aaa \in \QQ_p$ (where $\ZZ$ and $\QQ_p$ are thought of as canonically embedded into $G$). Each compact open coset of $G$ has a unique  form $g^z D_{r,a}$ for some $z \in \ZZ$.  Formally speaking, the domain of the computable copy of $\+ W(G)$  consists of natural numbers encoding  the triples $\la z, r, a\ra$ corresponding to such cosets; as before they will be identified with the cosets they denote. 

 To show that the  groupoid and meet    operations are   computable, note that we have     $gD_{r, a } = D_{r-1, a} g  $ for each $r \in \ZZ, a \in C_{p^\infty}$, and hence  
 %%%
 $ g^z D_{r, a } = D_{r-z, a}  g^z $ for each $z \in \ZZ$.
Given two cosets $g^v D_{r,a} $ and $ g^w D_{s,b}= D_{s-w, b}g^w$,  their composition is   defined iff $r=s-w$, in which case the result is $g^{v+w} D_{s,a+b}$. 
The inverse of $g^z D_{r,a}$ is $D_{r, -a}g^{-z} = g^{-z}D_{r-z, -a}$. 

To decide  the inclusion relation, note that we have $g^z D_{r,a} \sub g^w D_{s,b}$ iff $z=w$ and $D_{r,a} \sub D_{s,b}$, and otherwise, they are disjoint. Using this one can show that the meet operation is computable (by an argument that works in any computable meet groupoid $\+ V$): if $A_0, A_1 \in \+ V$,  $A_i \colon U_i \to V_i$, and $A_0,A_1$ are not disjoint, then  $A_0 \cap A_1$ is the unique $C \in \+ V$ such that $C \colon U_0 \cap U_1 \to V_0 \cap V_1$ and $C\sub A_0, A_1$.  Since $\+ W$ satisfies Condition (c) in \cref{def:comp_meet_groupoid}, and $\+V$  has no subgroups beyond the ones present in $ W$, we conclude that $\+ V$ is Haar computable. 
 \end{proof}
 
%\begin{remark} \label{rem:Zp} The profinite group $(\ZZ_p, +)$ is   computably t.d.l.c., because  the meet groupoid of cosets contained in $U_0$ is Haar computable. 
% One can also modify the proof  above in order  to obtain this  directly. Let now  $r$ range over $\NN$, and instead of the $\pi_r$ consider the natural projections $\eta_r \colon \ZZ_p \to C_{p^r}$. Let $D_{r,a}= \eta_r^{-1}(a)$ for $a \in C_{p^r}$. With these new notations, the computability of the operations and inclusion relation are obtained the same way as before, regarding  the  cyclic groups $C_{p^r}$ as subgroups of~$C_{p^\infty}$.  \end{remark}

\section{Equivalence of the two types of computable presentations}
We show that  a t.d.l.c.\ group $G$ has  a computable presentation in the sense of Def.\ \ref{def:main1}  iff $G$ 
 has  a computable presentation in the sense of  Def.\ \ref{Def2}. 
 
 We need some preliminaries.
For strings $\sss_0, \sss_1 \in \NN^*$ of the same length $n$, let $\sss_0 \otimes \sss_1$ denote the string $\tau$ of that  length such that $\tau(k)= \la \sss_0(k), \sss_1(k)\ra$ for each $k<n$ (where $\la ., .\ra$ is a computable  pairing function, such as Cantor's).   

\begin{lemma} \label{lem: decide inclusion}  Let $G$ be computably t.d.l.c.\ via a computable Baire presentation $([T], \Op, \Inv)$. Recall the set $E_T$ of minimal  code numbers for compact open sets from  \cref{def:E}.

\n (i)  There is  a computable   function  $ I \colon E_T \to E_T$     such that for each $u  \in E_T$, one   has  $\+ K_{I(u)} = (\+ K_u)^{-1}$. 

\n (ii) For $u,v,w \in E_T$  one can decide whether $\+ K_u \+ K_v \sub \+ K_w$. 
\end{lemma}
\begin{proof} (i) By Lemma~\ref{lem:lahm2} one can decide whether $\+ K_u\sub (\+ K_w)^{-1}$. The equality  $\+ K_u= (\+ K_w)^{-1}$  is equivalent to $\+ K_u\sub  (\+ K_w)^{-1}\lland \+ K_w\sub  (\+ K_u)^{-1}$. So one lets $I(u)$ be the least index $v$ such that this equality holds.

\n (ii) Let $\wt T$ be the tree of  initial segments of strings of the form $\sss_0 \otimes \sss_1$, where $\sss_0, \sss_1 \in T$ have  the same length. Then $\wt T$ is a c.l.c.\  tree, $[\wt T]$  is  naturally homeomorphic to $[T] \times [T]$, and $\Op$ can be seen as a computable function $[\wt T ] \to [T]$. Now one applies Lemma~\ref{lem:lahm2}. \end{proof}

%For   strings $\sss_i$, $i = 0,1$ with natural number entries,  of the same length $N\le \infty$,   $\sss_0 \otimes \sss_1$ denotes 
%%
% the string  of length $2N$ 
%%
%which alternates between $\sss_0$ and $\sss_1$. E.g. $\sss_0= (1, 3), \sss_1= (4,0)$, yields $(1,4,3,0)$.

%such that $\sss(2i+k)= \sss_k(i)$ for each $k \le 1$ and $ i < N$. 
  
  We will also need a computable  presentation of the topological group of permutations of $\NN$ based on a subtree of $\NN^*$. Define a computable tree without leaves by 
  
  \medskip

    $ \Tree {\S} = \{ \sss \otimes \tau \colon \sss , \tau \in \NN^*\lland$ 
    \bc  $ \sss, \tau \,  \text{are 1-1} \lland   \sss (\tau(k)) = k \lland \tau(\sss(i))= i \text{ whenever defined}\}$. \ec 
    
      %\bc $\sssl = |\tau|    \lland \sss, \tau \,  \text{are 1-1} \lland  \forall i,k < \sssl  \, [ \sss(i) = k \leftrightarrow \tau(k) = i]\}$. \ec 
    %
     A  string $ \sss \otimes \tau\in \Tree \S $ gives rise to a finite injection $\aaa_{ \sss \otimes \tau}$ on $\NN$, defined by   
\begin{equation} \label{eqn:inj} \aaa_{ \sss \otimes \tau}(r)=s  \text { iff } \sss(r)= s \,   \lor\,  \tau(s)=r.\end{equation}
   The paths of $\text{Tree}({\S})$ can be viewed as the  permutations of $\NN$, paired with their inverses: 
\bc   $[ \Tree {\S} ] =\{ f \otimes f^{-1} \colon \, f \text{ is permutation of } \NN\}$.   \ec
    The group operations on   $ \Tree {\S}$ are computable:   we have

    \begin{eqnarray*} (f_0 \otimes f_1) ^{-1}& = & f_1\otimes f_0 \\
(f_0 \otimes f_1) \cdot ( g_0 \otimes g_1) &  = & (f_0 \circ g_0) \otimes (g_1 \circ f_1).\end{eqnarray*} %It is not 
      
%    These operations of $\S$ are  given by computable functions on~$\text{Tree}({\S}) $, such as $\sss_0 \otimes \sss_1 \mapsto \sss_1 \otimes \sss_0$ for the inverse.   

For a closed subgroup $\wt G$ of $\S$, we write \bc  $\Tree { \wt G}= \{ \sss \in \Tree {\S}  \colon \, [ \sss]_{\Tree {\S} } \cap \wt G \neq \ES\}$. \ec Note that this is a   subtree  of $\Tree {\S}$ without leaves. We say that  $\wt G$ is computable if $\Tree {\wt G}$ is computable.

    \begin{thm}  \label{thm:main2} \ \\ A group   $G$  is computably t.d.l.c.\ via a  Baire~presentation (Def.\ \ref{def:main1})
$\LR$   

\hfill    $G$ is computably t.d.l.c.\   via a  meet groupoid (Def.\ \ref{Def2}). 
%\sasha{Seriously, if we keep `Haar computably t.d.l.c.\   via a  meet groupoid'
%you will have to promise to never complain about any terminology, ever. Just say groupoid computable, groupoid-computable presentation. And also, for the other one, maybe  $S_\infty$-computable presentation. (I do remember that `as a closed subgroup bla' was actually my terrible terminology.)}
%\andre{what a nonsense... Haar is not part of the def, you put it there. We say computably t.d.l.c. via X, and then indicate  what X is. Later on, 	``via X" can be omitted after proving the equivalences. Anything like ``groupoid computable, $\S$ computable etc is cryptic. }

From a presentation of $G$ of one type, one  can uniformly obtain a presentation of $G$ of  the other type.
\end{thm} 
 \begin{proof} 
 \lapf (This is the harder implication - if you don't want to read it skip to  Page~\pageref{haha}.) 
 
 We begin by defining an operator that, for Haar computable meet groupoids, is dual to the operation of sending $G$ to a   computable copy of    $\+ W(G)$  obtained above. 
 \begin{definition} \label{def:Gof} Given a   meet groupoid $\+ W$ with domain $\NN$, let $\wt G= \+ G_\text{comp} (\+ W)$ be the closed subgroup of $ \S$ consisting of elements $p$   that  preserve the meet operation of $\+ W$,  and satisfy  $p(A) \cdot B = p(A\cdot B) $ whenever $A \cdot B$ is defined.  \end{definition}

% Note that these $p$ are the automorphisms of the  structure obtained from $\+ W$ which, instead of composition,  for each $B$ has the partial unary operation $A \to A \cdot B$. 
 Recall    that the elements of $\S$ are not actually permutations, but paths on $ \Tree {\S}$ encoding   pairs consisting of a permutation and its inverse. However, if $p \in \Tree {\wt G}$, and $A \in \+ W$ is denoted by $i$, we will  suggestively  write $p(A)$ for the element of $\+ W$ denoted by  the first component of the pair of natural numbers encoded by $p(i)$.  We note that for  each  subgroup $U\in \+ W(G)$, the set $B=p(U)$ satisfies $B\cdot U = p(U)\cdot U= p(U\cdot U)= B$,  and hence is a left coset of~$U$. 
The following diagram displays the condition in the definition above in category terms as a commutative diagram.
 \[\renewcommand{\labelstyle}{\textstyle}  \xymatrixcolsep{9pc}\xymatrix { U\ar[r]|-{A} & V \ar[r]|-{B}&  W \\ 
 U' \ar[ur]|-{p(A)}   \ar[urr]|-{p(A\cdot B) } &   &  }\]

Now suppose that $\+ W$ is as in Definition~\ref{Def2}. Recall the    convention  that all t.d.l.c.\ groups are infinite. So the  domain of $\+ W$ equals  $\NN$,  and there is an isomorphism of meet groupoids $  \+ W \to \+W(G)$, which below we will use to identify $\+ W$ and $\+ W(G)$.    Define a group homomorphism  $\Phi \colon G \to \wt G$ by letting   $\Phi(g)$ be   the element of $\S$ corresponding to the  left action of~$g$, i.e. $A \mapsto gA$ where $A \in \+W(G)$.    Note that $\Phi$ is injective because the compact open subgroups form a neighbourhood basis of $1$:
if $g \neq 1$ then $g \not \in U$ for some compact open subgroup $U$, so that  $\Phi(g)(U) \neq U$.

%First we verify  that $\Phi \colon  G \cong \wt G$ where $\Phi(g)$ is  the left translation action of $g$, i.e. $A \mapsto gA$ for $A \in \+W(G)$.

\begin{claim} \label{cl:isomoG}  $\Phi \colon  G \cong \wt G$. \end{claim} \n  To show  that $\Phi$ is onto, let $p \in \wt G$. Since   \bc $\{p(U)\colon U \in \+ W(G)  \text{ is a subgroup}\}$  \ec   is a filter on $\+ W(G)$ containing a compact set, there is an element $g$ in its intersection.
  Then $\Phi(g)= p$: recall that for  each  subgroup $U\in \+ W(G)$, the set $B=p(U)$  is a left coset of $U$, and hence equals $gU$. So, if $A$ is a right coset of $U$, then $p(A)= p(U\cdot A)=  B\cdot A = gA$.
 
  To show that  $\Phi$ is continuous at $1$ (and hence continuous), note that a basis of neighbourhoods of the identity in $\wt G$ is given by the open sets
   \bc $\{ p\in \wt G \colon \forall i \le n \, [p(A_i) = A_i]\}$, \ec where $A_1, \ldots, A_n \in \+ W(G)$. Given such a set, suppose $A_i$ is a right coset of $U_i$, and let $U = \bigcap U_i$. If $g \in U$ then $gA_i = A_i$ for each $i$.

% By  the open mapping theorem for Hausdorff topological groups  states that a surjective continuous homomorphism of a  $\sss$-compact   group  onto a  Baire (e.g., a locally compact)  group  is an open mapping.  Each (separable) t.d.l.c.\ group is $\sss$-compact again by van Dantzig. So $\Phi$ is open.
The open mapping theorem for Hausdorff  groups says   that every surjective continuous homomorphism from a $\sss$-compact   group (such as  a t.d.l.c.\  group with a countable basis of the topology)   onto a Baire group is   open.  So $\Phi$ is open. This verifies the claim.

Using the assumption  that $\+ W$ is Haar computable, we now    show  that %$\wt G$ is computable t.d.l.c.\  as a closed subgroup of $\S$.  That is, 
$\Tree {\wt G} $ is c.l.c.\  as in Definition \ref{def:comploccompact}. 
The following claim will  be used to show that   $\Tree {\wt G}$ is computable. 
\begin{claim}  A  finite injection $\aaa $ on $\NN$ can be  extended to some $p \in \wt G$

\n   \hspace{3cm}   $\LR$   $B\cdot A^{-1}$ is defined whenever $\aaa(A)= B$, and  \bc $\bigcap  \{ B \cdot A^{-1} \colon \, \aaa(A) = B\} \neq \ES$. \ec \end{claim}
For right to left,   let $g$ be an element of the intersection. Then $gA = B\cdot A^{-1} \cdot  A = B=\aaa(A)$ for each $A \in \dom (\aaa)$.

For left to right,   suppose $p\in \wt G $ extends $\aaa$.  By \cref{cl:isomoG}, there is $g\in G$ such that  $p = \Phi(g)$.  Then $gA = p(A) = B$ for each   $A,B$ such that $\aaa(A)= B$. Such  $A,B$ are   right cosets of the same subgroup.  Hence $B\cdot A^{-1}$ is defined, and clearly  $g$ is in the intersection. 
This establishes the   claim.
 
 \
 
   By (\ref{eqn:inj}),      \[ S = \{\sss \otimes \tau \colon \,  
 \aaa_{ \sss \otimes \tau} \text{ can be extended  to some } p \in \wt G\}\]
 is a computable subtree of $\Tree \S$ without leaves, and $\wt G= [S]$.  Hence $S=\Tree {\wt G}$.
%%%%%%%
%\begin{claim}   $[\sss]_S$ is compact for each string $\sss \in S$ of length $2$. \end{claim} 
%\n To see this, let  $R$ be  the compact open subgroup of $G$ denoted by $0$. Let $\sss = (a,b)$, where $a $ denotes $A\in \+ W$, and $b$ denotes $B$. 
%Then $A$ is a left coset of $R$,  and $B$ is a right coset of $R$.
%   Note that  \begin{eqnarray*} [\sss]_S &=& \{ p \in \wt G \colon \, p(R)= A \lland p^{-1}(R)= B\}  \\ &= &\{ p \in \wt G \colon \, \Phi^{-1}(p) \in A\lland \Phi^{-1}(p^{-1})\in B\}. \end{eqnarray*}  
%   Since $\Phi$ is a homeomorphism  and $A,B$ are compact subsets of $G$, this settles the claim.  

%\n As before, let $W$ be the compact open subgroup of $G$ denoted by $0$. Suppose that $C=p(V)$ which is a left coset of a subgroup  $V$.   Note that $p(W \cap V)$ is a left coset of $W \cap V$ contained in $C$. There are  $k= |W : W \cap V|$ many such cosets. By hypothesis we can compute $k$, so we can search the computable structure $\+ W(G)$ until we find  all these cosets $\wt B_1, \ldots,\wt B_k$. We can also search for the corresponding left cosets $B_i$ of $V$ such that $B_i \supseteq \wt B_i$. The possible  images $p(C)$ are contained in the  set $\{B_1, \ldots, B_k\}$. Similarly, one obtains a bound  for the preimages. 
%
Claim~\ref{cl:fcuk} below will  verify  that  $S= \Tree {\wt G}$ is a    c.l.c.\  tree  as defined in \ref{def:comploccompact}. The following lemma does the main work.  Informally it says that given  some subgroup $U \in \+ W$,  if one  declares that $p\in \wt G$ has  a   value $L\in \+ W$  at $U$, then one  can compute for any $F \in \+ W$ the  finite set of possible values of $p$ at $F$. 
 %Recall that we identify an element of $\+ W(G)$ with the natural number denoting it.
\begin{lemma}[Effectively finite suborbits] \label{lem: sst} Suppose that   $U \in \+ W$ is a subgroup  and $L$ is a left coset of $U$. Let $F \in \+ W$. One can uniformly in    $U, L$ and $F$ compute a strong index for the  finite set  $\+ L=\{ p(F) \colon \, p \in  [S] \lland p(U) = L\}$.  
\end{lemma}
To see this, first one computes $V= F^{-1}\cdot F$,  so that $F$ is a right coset of the subgroup $V$. Next  one computes   $k=|U \colon U \cap V|$, the number of left cosets of $U \cap V$ in $U$.  Note that \bc  $\+ L_0= \{ p(U \cap V) \colon p \in  [S] \lland p(U) = L \}$ \ec  is the set of left cosets of $U \cap V$ contained in $L$.  Clearly this set  has size $k$. By searching $\+ W$ until all of its elements have appeared, one can compute a strong index for this  set.
 Next one computes a strong index  for   the set $\+ L_1$ of left cosets $D$ of $V$ such that $C \subseteq D$ for some $C \in \+ L_0$ (this uses that given  $C$ one can compute $D$). 
  Finally one outputs a strong {index}  for the set $\{ D \cdot F \colon \, D \in \+ L_1\}$, which equals $\+ L$.  This shows the lemma. 
 
%  \sasha{I do not know why but I found this little proof above hard to follow. I think it is correct. Maybe a few more detail here and there would help, it is pretty compressed. (It is basically a sequence of little claims, especially towards the end, each claim is elementary but there are just a bit too many of them to make it a comfortable read.)}
% \andre{ I will try to make the proof more readable, so that it's a pure pleasure to read it, only rivalled by having a Bacardi on the  deck of your 5 star-resort in the Bahamas, with a beautiful girl in a bikini close to you, and all this at sunset}

%\begin{claim} Suppose a coset $C$ is denoted by $n$ where $n>0$. One can compute a bound $h(n)$ on the  (numbers denoting) possible images $p(C)$,   and the preimages $p^{-1}(C)$,  of any  $p\in \wt G$. \end{claim} 

We   make  the assumption that $0$ denotes  a subgroup $U$  in $\+ W$. This does not      affect  the uniformity statement of the theorem: otherwise we can search   $\+ W$ for the least $n$ such that $n$ is a subgroup, and then work with a new copy of $\+ W(G)$ where~$0$ and $n$ are swapped.  

   \begin{claim} \label{cl:fcuk} There is a computable binary   function $H$ such that,  if         $\rho \in S$, then  $\rho(i) \le H(\rho(0), i)$ for each $i < |\rho|$.    \end{claim}

\n    Let    $F$ be the coset denoted by $k$. Let $\rho(0)= \la a_0,a_1\ra$ and let  $L_r$ be the coset denoted by $a_r$, $r= 0,1$.
  Applying  Lemma~\ref{lem: sst} to $U, L_r, F$, one can compute     $H(\sss, i) $ as the greatest pair $\la b_0, b_1\ra$ such that   $b_r$ denotes  an element of $\{ p(F) \colon \, p \in  [S] \lland p(U) = L_r \}$ for $r=0,1$.    
  
  \medskip 

 %We define a    computably t.d.l.c.\   closed subgroup $\wt G$ of $\S$. 
%We begin by defining an operator that for Haar computable meet groupoids is dual to the operation of sending $G$ to a   computable copy of    $\+ W(G)$  obtained above. 
 
\label{haha}

  \rapf We build a Haar computable copy $\+ W$ of the meet groupoid $\+ W(G)$ as in Definition~\ref{Def2}. By  Lemma~\ref{lem: decide inclusion}, one can decide whether $u \in E_T$ is the code number of a subgroup (\cref{defn: str index compact}).   Furthermore,  one can decide whether $B= \+ K_v $ is a left  coset of  a subgroup $U= \+ K_u$: this holds iff   $BU \sub B$ and $BB^{-1} \sub U$, and the  latter two conditions are decidable by Lemma~\ref{lem: decide inclusion}. Similarly, one can decide whether $B$ is a right coset of $U$.
 
It follows that the set $\{ u \in E_T \colon \, \+ K_u \text{ is a  coset}\}$ can be  obtained via   an  existential quantification over a computable binary relation (in other words, $V$ is recursively enumerable). Hence,  by  a basic fact of computability theory,    there is   computable 1-1 function $\theta$ defined on an initial segment of $  \NN- \{0\}$ such that the range of $\theta$ equals this  set. Write $A_n = \+ K_{\theta(n)}$ for $n>0$, and $A_0 = \ES$.

The domain of   $\+ W$   is all of $\NN$.     By  \cref{lem: comp index tree}  the intersection operation on $\+ W$ is computable, i.e., there is a computable binary function $c$ on $\NN$ such that $A_{c(n,k)} =A_n \cap A_k$. Next, given $n,k \in \NN-\{0\}$ one can decide whether $A_n$  is a right coset of the same  subgroup   that  $A_k$  is a left coset of. In that case,  one can compute the number  $r$  such that  $A_r= A_n\cdot A_k$: one uses that  $A_r$ is the unique coset  $C$ such that \bi \item[(a)] $A_nA_k \sub C$,  and   \item[(b)]$C$ is a right coset of the same subgroup that  $A_k$ is a right coset~of. \ei
 For subgroups $U,V$, one can compute  $|U \colon U \cap V|$ by finding in  $\+ W$ further  and further  distinct left cosets of $U \cap V$ contained in $U$,  until their union reaches $U$. The latter condition is decidable. \end{proof}
 
 \begin{definition} \label{def:Wcomp} Given a computable Baire presentation $G$, by $\+ W_\text{comp} (G)$ we denote the computable copy of $\+ W(G)$ with domain $\NN$ obtained in the proof above. \end{definition}

\begin{cor}  \label{cor:action} In this context, the   left and right actions $[T] \times \NN \to \NN$, given by   $(g,A) \mapsto gA$ and  $(g,A) \mapsto Ag$,  are computable.  \end{cor}

\begin{proof} For the left action, we use an   oracle Turing machine  that  has as an  oracle a path $g$ on $[T]$, and as an input an  $A \in \+ W$. If $A$ is  a left coset of a subgroup $V$, it outputs     the left  coset $B$  of~$V$ such that it can   find a string   $\sss \prec g$ with  $[\sss]_TA \sub B$.   

For the right action use that  $Ag = (g^{-1}A^{-1})^{-1}$ and inversion is computable both in $G$ and in $\+ W_\text{comp} (G)$.  \end{proof}

 Recall from the introduction  that  $\mathrm{Aut}(T_d)$ is the group of automorphism of the undirected tree   $T_d$  where each vertex has degree $d$.
 \begin{example}  \label{ex:Td}Let $d\ge 3$. The t.d.l.c.\  group $G=\text{Aut}(T_d)$ has a computable Baire presentation.   \end{example}

   \begin{proof} Via  an effective encoding  of the vertices of $T_d$ by the natural numbers, we can view $G$ itself as a closed subgroup of $\S$. A finite injection $\aaa$ on $T_d$ can be extended to an automorphism of $T_d$ iff   it preserves   distances, which is a decidable condition. Each $\eta \in\mathit{Tree}(\S)$ corresponds to   an injection on $T_d$ via (\ref{eqn:inj}). 
  So we can decide whether $[\eta]_{\Tree G }=[\eta]_{\Tree {\S}} \cap G \neq \ES$.  
 Clearly   $[\eta]_{\mathit{Tree}(G)}$ is compact for \emph{every}  such nonempty string $\eta$.

To see that $\mathit{Tree}(G)$ is c.l.c., note that   if $\sss\in \mathit{Tree}(G)$ maps $x\in T_d$ to $y\in T_d$, then every extension $\eta \in \mathit{Tree}(G) $ of $\sss$ maps   elements  in $T_d$ at distance $n$ from $x$ to elements  in $T_d$ at distance $n$ from $y$, and conversely.  
 This yields a computable bound $H(\sss, i)$ as required in (3) of Def.~\ref{def:comploccompact}. \end{proof}

 %}

\section{Algorithmic properties of objects associated with a   t.d.l.c.\ group}
\subsection{The modular function is computable} In Subsection~\ref{Assoc.Comp} we discussed  the modular function $\Delta \colon G \to \RR^+$.  As an   application of \cref{cor:action}, we show that  for any computable presentation, the modular function   is computable.   \begin{cor}  \label{cor:delta} Let $G$ be computably t.d.l.c.\ via  a Baire presentation $([T],  \Op, \Inv)$. Then the  modular function $\Delta \colon [T] \to \QQ^+$ is computable. \end{cor}
\begin{proof}   Let $V \in \+ W$ be any subgroup.
Given $g \in [T]$,  using \cref{cor:action} compute $A=g V$. Compute $U \in \+ W$ such that  $A $ is a right coset of $U$, and hence $A = Ug$. For any left Haar measure $\mu$ on $G$, we have \bc $\Delta(g)= \mu(A)/\mu(U)=  \mu(V)/\mu(U)$.  \ec  By \cref{rem:Haar computable} we can choose $\mu$ computable; so this suffices to determine $\Delta(g)$.  
 \end{proof}

\subsection{Cayley-Abels graphs are computable}
   Let $G$ be a t.d.l.c.\ group that is   compactly generated, i.e., algebraically generated by a compact subset. Then there is a compact open subgroup $U$, and a     set $S = \{ s_1, \ldots, s_k\} \sub G$ such that $S = S^{-1}$ and $U \cup S$ algebraically generates $G$. The \emph{Cayley-Abels graph} 
   \bc $  \Gamma_{S,U}= (V_{S,U}, E_{S,U})$ \ec of $G$ is given as follows. The  vertex set $V_{S,U}$ is the set $L(U) $ of left cosets of $U$, and the edge relation is   
\bc $E_{S,U} = \{ \la gU, gsU \ra \colon \, g \in G, s \in S\}$. \ec
Some background and      original references are given  in Section 5 of \cite{Willis:17}. For more detailed background see Part~4 of \cite{Wesolek:18}, or \cite[Section~2]{Kron.Moller:08}.  If $G$ is discrete (and hence finitely generated), then $\Gamma_{S,\{1\}}$ is the usual Cayley graph for the generating set~$S$.  Any two  Cayley-Abels graphs of $G$  is  are quasi-isometric.  See \cite[Def.~3]{Kron.Moller:08} or \cite{Wesolek:18} for the formal definition.

%\sasha{computably? that would be reasonably cool (if it is not hard to see). I forgot whether this was even the case for discrete, but it surely must be. I think Bakh checked it formally some stage long ago} \andre{yea that would be interesting and may be easy to determine by reading Wesolek. I can ask bakh}

%\sasha{You'll have to tell me what this graph thing is, I've never heard of it. It seems to be very specific to the field of tdlc groups, so maybe we should say a few words about it.}
\begin{thm} \label{prop:CA graph} Suppose that $G$ is computably t.d.l.c.\ and compactly generated.   \bi \item[(i)] Each Cayley-Abels graph $\Gamma_{S,U}$ of $G$ has a computable copy $\+ L$. 
%Given a Haar computable copy $\+ W$ of the meet groupoid $\+ W(G)$, one  can obtain $\+ L$ effectively from    $U \in \+ W$ and the left cosets $C_i = s_iU$, where $\{ s_1, \ldots, s_k\} $ is as above.   

\item[(ii)] If $\Gamma_{T,V}$ is another Cayley-Abels graph obtained as above, then $\Gamma_{S,U}$ and  $\Gamma_{T,V}$ are computably quasi-isometric. 

%\item[(iii)] Given a computable Baire presentation of $G$ based on a tree $[T]$, let $\+ W= \+ W_{\text{comp}}(G)$  be the computable copy of its  meet groupoid  as in  \cref{def:Wcomp}. Then  the left action  $[T] \times \+ L  \to \+ L$ is also computable.  
 \ei \end{thm} 
\begin{proof} %  Let  be  As usual we assume the domain of $\+ W$ is $\NN$.
 (i) For the domain of the computable copy $\+ L$,   we take the computable set of left cosets of $U$.    We show that the edge relation is first-order definable from the parameters in such a way that it can be verified to be computable as well. 

Let $V_i = C_i \cdot   C_i^{-1}$ so that $C_i$ is a right coset of $V_i$. Let $V =  U \cap \bigcap_{1 \le i \le k} V_i$. To first-order  define $E_\Gamma$ in $\+ W$ with the given parameters, the idea is to replace the elements  $g$ in the definition of $E_\Gamma$ by left cosets $P$ of $V$, since they are sufficiently accurate approximations to $g$.
 It is easy to verify that $\la A, B \ra \in E_\Gamma$ $\LR$ 
 \bc $\ex i \le k \ex P \in L(V) \ex Q \in L(V_i) \, [ P \sub A \lland P \sub Q \lland B = Q \cdot C_i]$,  \ec
 where $L(U)$ denotes the set of left cosets of a subgroup $U$: For the implication ``$\LA$", let $g\in P$; then we have $A = gU$ and $B = gs_i U$.  For the implication ``$\RA$", given $A = gU$ and $B= gs_i U$, let $P\in L(V)$ such that $g \in P$. 
 
 We verify that the edge relation  $E_\Gamma$ is computable. Since $\+ W$ is Haar computable, by the usual enumeration argument we can obtain  a strong index for the set of  left cosets of $V$ contained in $A$. Given  $P$ in this set and $i \le k$, the left coset  $Q= Q_{P,i}$ of $V_i$ in the expression above is unique and   can be  determined effectively. So we can test whether $\la A, B \ra \in E_\Gamma$ by trying all $P$ and all $i\le k$ and checking whether $B = Q_{P,i} \cdot C_i$. 
%
% It is clear from the argument that we obtained $\+ L$ effectively from the parameters $U$ and $C_i$.  
 
 \smallskip
 
\n (ii) (Sketch) First suppose that $V \sub U$. There is a computable map $\psi \colon L(U) \to L(V)$ such that $\psi(A) \sub A$. The proof of \cite[Thm.\ 2$^+$]{Kron.Moller:08} shows that $\psi \colon \Gamma_{S,U} \to \Gamma_{T,V}$ is a quasi-isometry. In the general  case, let $R\sub G$ be a finite symmetric set such that $(U \cap V) \cup R$ algebraically generates $G$. There are  computable quasi-isometries $\phi \colon \Gamma_{S,U} \to \Gamma_{R, U \cap V}$ and  $\psi \colon \Gamma_{T,V} \to \Gamma_{R, U \cap V}$ as above. There is a computable quasi-isometry $\theta: \Gamma_{R, U \cap V} \to  \Gamma_{T,V}$: given a vertex $y\in L(U \cap V)$, let    $x= \theta(y)$ be a vertex in $L( V)$ such that  $\psi(x)$ is at distance at most $c$ from $y$, where $c$ is a constant for  $\psi$  as above.   Then $\theta \circ \phi$ is a quasi-isometry as required. 
%\smallskip
%
%\n 
%  (iii)    follows immediately from  Theorem~\ref{prop: comp isom}(ii). 
\end{proof}

 \subsection{Algorithmic properties of the scale function} \label{s:scale}
The scale function $s \colon G \to \NN^+$ for  a t.d.l.c.\ group~$G$ was   introduced by  Willis~\cite{Willis:94}. Recall that  for a compact  open subgroup~$V$ of $G$ and an element $g\in G$ one defines $m(g,V) = |V^g\colon V \cap V^g| $,   and   
\bc $s(g)= \min \{ m(g,V)\colon \,  V \text{ is a compact open subgroup}\}$. \ec  
 Willis    proved  that the scale function is continuous,  where $\NN^+$ carries the discrete topology. He  introduced the relation that  a compact open subgroup $V$ is  \emph{tidy for~$g$}, and showed that this condition is equivalent to being   minimizing for $g$ in   the sense that $s(g)= m(g,V)$. M\"oller  \cite{Moeller:02} used graph theoretic methods to   show  that  $V$ is minimizing for $g$ if and only if  $ m(g^k, V) = m(g,V)^k$ for each $k\in \NN$. He also derived the ``spectral radius formula": for any compact open subgroup $U$, one has $s(g)= \lim_k m(g^k, U)^{1/k}$. 
%Note that  the scale function is constant $1$ when there is a normal compact open subgroup (e.g. when $G$ is discrete, or abelian, or profinite).  
 
 The following example is well-known (\cite[Example 2]{Willis:17}); we include it to show that our framework is adequate as a general background for case-based approaches to computability for  t.d.l.c.\ groups  used in earlier works.
 \begin{example}[with Stephan Tornier] For $d \ge 3$, the scale function on $\Aut(T_d)$ in the computable presentation of \cref{ex:Td} is computable. \end{example}
 \begin{proof}
 An automorphism $g$ of $T_d$ has exactly one of three types (see  \cite{Figa.Nebbia:91}):  
 \begin{enumerate}
\item $g$ fixes a vertex $v$: then $s(g)= 1$ because $g$ preserves the stabilizer of $v$, which is a    compact open subgroup. 

\item $g$ inverts an edge: then $s(g)= 1$ because $g$ preserves     the set-wise stabilizer of the set of  endpoints of this edge.

\item $g$ translates along a geodesic (a subset of $T_d$ that is a homogeneous  tree of degree $2$): then $s(g) = (d-1)^\ell $ where  $\ell$ is the  length. 
To see this, for $\ell=1$   one uses as a minimizing subgroup  the compact open subgroup of automorphisms that fix two given adjacent vertices on the axis. For $\ell>1 $ one uses that $s(r^k)= s(r)^k$ for each $k$ and $r \in \Aut(T_d)$; see again~\cite{Willis:94}. 
\end{enumerate}
The oracle Turing machine, with  a path corresponding to   $g \in \Aut(T_d)$ as an oracle, searches in parallel  for a witness to (1), a witness to (2), and a sufficiently long piece of the axis in (3) so that the shift becomes ``visible". It then outputs the corresponding value of the scale.
 \end{proof}

 For  the rest of this section, fix a  computable Baire presentation $([T], \Op, \Inv)$  of a t.d.l.c.\ group $G$ as in Def.\ \ref{def:main1}.   Let $\+ W = \+ W_{\text{comp}}(G)$ be the   Haar computable copy of $ \+ W(G)$ given by \cref{def:Wcomp}. Recall that  the domain of~$\+ W$ is $\NN$.  Via $\+ W$ we can  identify compact open cosets of $G$ with natural numbers. 
% Here we use capital letters for the elements of $D$ (the minimal code numbers of compact open cosets of $G$). 
 The following is immediate from  \cref{cor:action}.
\begin{fact}  \label{fact:mg}The function $m \colon [T]  \times \+ \NN \to \NN$ (defined to be $0$ if the second argument  is not a subgroup) is computable. \end{fact}

% So to construct a computable t.d.l.c.\ group $G$ where the scale is not computable requires to $G$ to be ``genuinely" t.d.l.c.
 
 It is of   interest to study whether   the scale function, seen as a function  $s\colon [T]\to \NN$,   is computable in the  sense of   \cref{def:computable 1}.  We note that neither M\"oller's  spectral radius formula, nor the tidying procedure of Willis (see again \cite{Willis:17}) allow  to compute the scale  in our sense.
% Several papers assert that the scale is computable for particular groups (usually using an ad-hoc framework); see \cite{Willis:17} for an overview. %We provide  a sequence of  examples where it is.

The scale  is computable if and only if one can algorithmically decide whether a subgroup is mimimizing:
\begin{fact} The scale function  on $[T]$ is computable  $\LR$ the  following function $\Phi$ is computable in the sense of Def.\ \ref{def:computable 1}: if   $g \in [T] $ and $V$ is a compact open subgroup of $G$, then $\Phi(g,V)=1$ if $V$ is minimizing for~$g$;  otherwise $\Phi(g,V)=0$.      \end{fact}  
\begin{proof} \rapf An  oracle Turing machine with oracle $g$ searches for the first~$V$ that is minimizing  for $g$, and outputs $m(g,V)$. 

\lapf For  oracle $g$, given input $V$ check whether $m(g,V) = s(g)$.  If so output $1$, otherwise $0$. \end{proof}

We next provide a fact restricting  the complexity of the scale function.  We     say that a function $\Psi: [T] \to \NN$ is \emph{computably approximable from above} if there is a computable function $\Theta: [T] \times \NN \to \NN$ such that $\Theta(f ,r) \ge \Theta(f ,r+1)$ for each $f\in[T],r\in \NN$, and   \bc $\Psi(f)= k $ iff  $\lim_r \Theta(f ,r) =k$.  \ec
 \begin{fact} The scale function   is   computably approximable from above. \end{fact} 
 \begin{proof} Let $\Theta(f,r)$ be the minimum value of $m(f,s)$ over all $s \le r$.  
 %
%  Oracle Turing machine computing $\Phi(g,t)$ looks for the least  compact open $U$ by code number $u\le t$ such 
%that  $\forall n \le t \, m(U,g^n) = m(U,g)^n$, and outputs $m(U,g)$. Eventually some tidy $U$ is found, so from then on the output is $s(g)$. 
\end{proof}

\section{Closure properties of the class of computably t.d.l.c.\ groups} \label{s:closure}
All computable presentations in this section will be Baire presentations (see \cref{def:main1}), and we will usually   view a t.d.l.c.\  group $G$  concretely as  a    computable Baire presentation. Extending the previous notation in the setting of   closed subgroups of $\S$, by $\Tree G$ we denote the c.l.c.\  tree underlying this computable Baire presentation.
The following is immediate. 
\begin{fact}[Computable closed subgroups]\label{fact:immediate} Let $G$ be a computably   t.d.l.c.\ group. Let $H$ be a closed subgroup of $G$ (so that  $\Tree H$  is a subtree of $\Tree G$). %Suppose that   $\Tree H$ is  c.l.c.\   
  Then $H$ is computably t.d.l.c.\ via the Baire presentation based on the $\Tree H$ (which is c.l.c.),  with   the operations of $G$  restricted  to $H$. \end{fact}    
For instance, consider the closed subgroups  $U(F)$ of $\Aut(T_d)$, where $d \ge 3$ and $F $ is a subgroup of  $S_d$, introduced by Burger and Mozes~\cite{Burger.Mozes:00}. By \cref{ex:Td} together with the preceding fact, each group $U(F)$ is computably t.d.l.c.

For another  example, consider the computable Baire presentation of $\SL_2(\QQ_p)$   given by   \cref{prop: SL2}. Let $S$ be the c.l.c.\  subtree of $T$ whose  paths describe    matrices of the form  $\begin{pmatrix} r & 0\\0 & s \end{pmatrix}$ (so that $s= r^{-1}$). This yields   a computable Baire presentation of the  group $(\QQ_p^*, \cdot)$. 

\begin{example} \label{ex:GL2} For each prime $p$ and $n \ge 2$, the   group % $\ZZ \ltimes \QQ_p$ and
 $\GL_n(\QQ_p)$ is  computably t.d.l.c.
\end{example} 
\begin{proof} 
We   employ  the embedding $F\colon GL_n(\QQ_p) \to SL_{n+1}(\QQ_p)$ which extends  a matrix $A$ to the matrix $B$ where the new row and new column vanish except for the diagonal element (which    necessarily  equals $(\det A)^{-1}$).  Clearly there is a c.l.c.\ subtree $S$ of    the c.l.c.\  subtree of $T$  in  \cref{prop: SL2} for $n+1$ such that $[S] = \range (F)$. Now we apply \cref{fact:immediate}. \end{proof}

A further construction staying within the class of t.d.l.c.\ groups is the semidirect product based on a continuous action. In the effective setting, we use  actions that are computable in the sense of Section~\ref{s:comp notions}.
% For   computable actions in the more  general context of Polish groups see  \cite{MeMo}.
\begin{prop}[Closure under computable semidirect products] \label{prop:sdprod}
Let $G, H$ be computably  t.d.l.c.\ groups. Suppose $\Phi \colon G \times H \to H $ is a computable function   that specifies an action of $G$ on $H$ via topological automorphisms. Then the topological semidirect product $L=G \ltimes_\Phi H$ is computably t.d.l.c. \end{prop}
%\sasha{We need to define what we meen by a computable topological action. We should be consistent with \cite{MeMo}; I think there Antonio and I assumed the actions have to be effectively open as well. (Perhaps in our case this is true as a theorem? Not clear.}
%\andre{I meant computable function in the sense of Section 6, have clarified this. I don't require effectively open, probably it's automatic, but it's not needed}

\begin{proof} Let $T$ be the tree obtained by   pairing corresponding components of  strings of the same length from the trees of $G$ and $H$, i.e.\ 
\bc $T= \{ \sss \otimes \tau \colon \, \sss \in \Tree G \lland \tau \in \Tree H\}$. \ec
It is clear that $T$ is a c.l.c.\ tree. Via the  natural bijection \bc $[T] \to [\Tree G] \times [ \Tree H]$, \ec  one  can write elements of $L$ in the form $\la g, h \ra$ where $g\in [\Tree G]$ and $h \in [\Tree H]$. 

By the standard  definition of semidirect product (\cite[p.\ 27]{Robinson:82}), writing the operations for $G$ and $H$ in the usual group theoretic way, we have 
\begin{eqnarray*} \Op ( \la g_1, h_1 \ra, \la g_2, h_2\ra) &=& \la g_1 g_2, \Phi(g_2, h_1) h_2\ra \\
\Inv( \la g, h \ra)  & = & \la g^{-1}, (\Phi (g^{-1}, h))^{-1}\ra . \end{eqnarray*}
This shows that  $\Op$ and $\Inv$ are computable, and hence   yields a computable Baire presentation $([T], \Op, \Inv)$ for $L$. 
\end{proof}

  The next two closure properties are proved in the underlying paper \cite[Section 11]{Melnikov.Nies:22}.
  For local direct products see Wesolek  \cite[Def.\ 2.3]{Wesolek:15}.
\begin{prop}[Prop 11.5 in \cite{Melnikov.Nies:22}]  \label{prop: AAA} Let $(G_i)\sNp i$ be  computably t.d.l.c.\ groups uniformly in $i$,  and for each $i$  let $U_i$ be a compact open subgroup of $G_i$, uniformly in $i$.  Then $G=\bigoplus \sNp i  (G_i, U_i) $ is computably t.d.l.c. \end{prop}

 The hardest one is the closure under quotients by computable closed normal subgroups.
 \begin{thm}[Thm.  11.11 in \cite{Melnikov.Nies:22}] \label{thm:closure normal} Let $G$ be computably t.d.l.c. Let $N$ be a closed normal subgroup of $G$ such that $\Tree N$ is a computable  subtree of $\Tree G$. Then $G/N$ is computably t.d.l.c. \end{thm}

 \begin{example} For each prime $p$ and each $n \ge 2$, the group $\PGL_n(\QQ_p)$ is computably t.d.l.c. \end{example} 
 \begin{proof} 
In \cref{ex:GL2} we obtained a    computable Baire presentation $(T,   \Op, \Inv) $ of $\GL_n(\QQ_p)$.  In this presentation,  the centre $N$ of $\GL_n(\QQ_p)$ is given by the diagonal $(n+1) \times (n+1)$ matrices such that the first $n$ entries of the diagonal agree. So clearly $\Tree N$ is a  computable  subtree  of the tree $S$ in \cref{ex:GL2}. Hence we can apply \cref{thm:closure normal}.
% The centre $N$ of $GL_2(\QQ_p)$ consists of matrices of the form  $\begin{pmatrix} q^{2} & 0\\0 & 1   \end{pmatrix}  \begin{pmatrix} q^{-1} & 0\\0 & q   \end{pmatrix}$, where $q \in\QQ_p^*$; the left matrix encodes an element of $\QQ_p^*$, and the right matrix is in $SL_2(\QQ_p)$.  Using this  one verifies that $\Tree N$ is a computable subtree of $T$. 
 \end{proof}

\section{Uniqueness of computable  presentation} \label{s:autostable} 
As discussed in Subsection~\ref{ss:auto},  a countable structure is called  {autostable}  if  it has a computable copy, and all  its computable copies are computably isomorphic.  We adapt this notion to the present setting. 

\begin{definition} A  computably t.d.l.c.\ group $G$ is called \emph{autostable} if for any two computable Baire presentations of $G$, based on    trees  $T, S\sub \NN^*$, there is a computable group homeomorphism $\Psi \colon [T] \to [S]$. Note that $\Psi^{-1}$ is also  computable  by \cite[Cor 9.3]{Melnikov.Nies:22}.  \end{definition}

 %
%For  abelian profinite groups, the notion of autostability used in \cite{Pontr} is equivalent to our definition. This follows from   the proofs of Prop.~\ref{prop:proctotdlc} and Prop.~\ref{lem:proc}, which show that in the  abelian case the correspondence between Baire presentations and   procountable  presentations is uniform and   witnessed by uniformly obtained group-isomorphisms between these presentations.  The first author  \cite[Cor.\ 1.11]{Pontr}     characterizes  autostability for  abelian compact pro-$p$ groups  given by  computable procountable   presentations with effectively finite kernels:  such a group is autostable iff its Pontryagin - van Kampen dual is autostable.  
% For instance,    $(\ZZ_p, +)$ is autostable  because  its dual  is the Pr\"ufer group $C_{p^\infty}$, which is easily seen to be autostable as a countable structure (see the proof of \cref{thm:autostable} below).

We now provide a criterion for autostability, and show its usefulness  through various examples. 
\begin{prop} \label{thm:compCrit} A  computably  t.d.l.c.\ group $G$  is autostable $\LR$  
 any two Haar computable copies of   its meet groupoid $\+ W(G)$ are computably isomorphic. \end{prop}
We will only apply the implication ``$\LA$". However,  the converse implication  is interesting on its own right because it shows that our notion of autostability is independent of whether we use computable Baire presentation, or  computable presentations based on    meet groupoids.
\begin{proof} 
   See \cite{Melnikov.Nies:22}, proof of Criterion 12.2.
   \end{proof}

\begin{theorem}  \label{thm:autostable} The computably t.d.l.c.\ groups   $\QQ_p$ and $\ZZ \ltimes \QQ_p$ are   autostable. \end{theorem}
\begin{proof}   
  In  \cref{ex:Qp} we obtained  a   Haar computable copy  $\+ W$ of  the meet groupoid $\+ W(\QQ_p)$. Recall that the elements of $\+ W$ are given as cosets  $D_{r,a}=\pi_r^{-1}(a)$ where $r\in \ZZ$, $\pi_r \colon \ZZ_p \to C_{p^\infty}$ is the canonical projection with kernel  $U_r =p^{r} \ZZ_p$,  and $a \in C_{p^\infty}$.

By the criterion above, it  suffices to show that any Haar computable copy $\wt {\+ W}$ of $\+ W(\QQ_p)$ is computably isomorphic to $\+ W$. 
  By hypothesis on $\wt {\+ W}$  there is an isomorphism $\Gamma \colon \+ W \to \wt {\+ W}$. Let $\wt U_r = \Gamma(U_r)$ for $r\in \ZZ$.  We will construct a \emph{computable} isomorphism $\Delta\colon \+ W \to \wt {\+ W}$  which agrees with $\Gamma $ on the set   $\{U_r\colon \, r\in \ZZ\}$. 
First we show that from $r$ one can   compute the subgroup $\wt U_r \in \wt {\+ W}$. 

\bi \item[(a)] If $\wt U_r$ has been determined, $r\ge 0$, compute  $\wt U_{r+1}$ by searching for  the unique subgroup in $\wt {\+ W}$   that has index $p$ in $\wt U_r$.

\item[(b)] 
If $\wt U_r$ has been determined, $r\le 0$, compute  $\wt U_{r-1}$ by searching for  the unique subgroup in $\wt {\+ W}$ such that $\wt U_r$ has index $p$ in it.  \ei

The shift homeomorphism  $S\colon \QQ_p \to \QQ_p$ is defined by $S(x)= p x$.  Note that $B \to S(B)$ is   an automorphism of the meet groupoid $\+ W$.  Using  the notation of Example~\ref{ex:Qp} (recalled above), for each $\aaa \in \QQ_p, r \in \ZZ$,  one has $\pi_{r+1}(S(\aaa))=\pi_r(\aaa)$, and hence  for each $a \in C_{p^\infty}$, \begin{equation} \label{eqn:SSS} S(D_{r,a})= D_{r+1, a}. \end{equation}   
We show that $S$  is definable    within $\+ W$ by an existential formula using    subgroups~$U_r$ as parameters. Recall that given a meet groupoid $\+ W$, by $L(U)$ we denote  the set of left cosets of a subgroup $U$. For $D \in L(U_r)$ we write   $D^k$ for $D \cdot \ldots \cdot D$ (with $k$~factors), noting that  this is defined,  and in $L(U_r)$.%

\begin{claim} Let $B \in L(U_r)$ and $ C\in L(U_{r+1})$. Then
 \bc $C = S(B) \LR  \ex D \in L(U_{r+1})\, [D \sub B \lland D^p = C]$. \ec \end{claim}
\n  \lapf If $x \in C$ then $x = py$ for some $y \in B$, so $x \in S(B)$.    So $C \sub S(B)$  and hence $C= S(B)$ given that $S(B)\in L(U_{r+1})$. 
 
\rapf Let $x\in C$, so $x = S(y)$ for some $y \in B$. Let $y \in D$ where $D \in L(U_{r+1})$. Then $D\sub B$. Since $D^p \cap C \neq \ES$,   these two (left) cosets of $U_{r+1}$ coincide.  This shows the claim.

  We  use this to show that the function   $\wt S=\Gamma \circ S \circ \Gamma^{-1} $ defined on $\wt {\+ W}$ is computable. Since $\Gamma( U_r)= \wt U_r$,  ($r \in \ZZ$),   $\wt S $ satisfies   the claim when replacing the $U_r$ by the $\wt U_r$.  Since the meet groupoid $\wt {\+ W}$ is computable,  given $B \in \wt {\+ W}$,     one can search $\wt {\+ W}$ for a witness $D\in L(\wt U_{r+1})$ as on the right hand side, and then output $C= \wt S(B)$. So the function  $\wt S$ is computable.

%We will view  $C_{p^r}$  as a subgroup of $C_{p^\infty}= \ZZ[1/p]/\ZZ$, 
%namely, $C_{p^r}= p^{-r} \ZZ/\ZZ$. 
We build the computable isomorphism $\Delta \colon \+ W \to \wt {\+ W}$ in four phases.    The first three phases build a computable   isomorphism $L( U_0) \to L(\wt U_0)$, where $L(\wt U_0) \sub \wt {\+ W}$ denotes the group    of left cosets of $\wt U_0$. (This group is isomorphic to $C_{p^\infty}$, so this amounts to defining a computable isomorphism between two computable copies of $C_{p^\infty}$.)  The last phase extends this isomorphism to all of $ \+  W$, using that $\wt S$ is an automorphism of $\wt {\+ W}$. 

%\sasha{This proof is cool, I like it. It totally makes sense. I think it might work whenever we begin with a cc discrete group and extend it by $\mathbb{Z}_p$ (being the kernel, i.e., in the homological terms). It may or may not be true, however, you do not really use that the factor is $C_{p^\infty}$ until here, right? You did not really need $S$ to be an automorphism, you needed a bit less than that -- that the formula works. }
%
%\andre{yes it looks likely but I want to keep the examples concrete...
%I've put your problems suggestions at the end for further consideration}

For $q \in \ZZ[1/p] $ we write  $[q]= \ZZ + q \in C_{p^\infty}$.   We     define $\wt D_{r,[q]}= \Delta(D_{r,a})$ for $r \in \ZZ, q \in \ZZ[1/p] $

\bi \item[(a)] Let $\wt D_{0, [p^{-1}]}$ be an element of order $p$ in $L(\wt U_0) $. 
\item[(b)] Recursively,  for  $m>0$ let  $\wt D_{0, [p^{-m}]}$ be an element of order $p^m$ in $L(\wt U_0) $ such that $(\wt D_{0, [p^{-m}]})^p = \wt D_{0, [p^{-m+1}]}$.   
 \item[(c)]  For $a = [kp^{-m}]$ where $0 \le k <p^m$ and $p  $ does not divide $k$,  let $\wt D_{0, a} = (\wt D_{0, [p^{-m}]})^k$.
\item[(d)] For $r \in \ZZ -\{0\}$ let  $\wt D_{r,a} =\wt S^r (\wt D_{0,a})$.\ei 
One can easily verify that $\Delta \colon \+ W \to \wt {\+ W}$ is  computable  and preserves the   meet groupoid operations.   To  verify that $\Delta$ is onto, let  $B \in \wt {\+ W}$. We have $B \in L(\wt U_r)$ for some $r$.   There is a least $m$ such that $B= (\wt D_{r, [p^{-m}]})^k$ for some $k< p^m$.  Then  $p$ does not divide $k$, so $B= \wt D_{r, [kp^{-m}]}$. 

 \medskip
 
 \n   \emph{We next   treat   the   case of $G=\ZZ \ltimes \QQ_p$.} Let $\+ V$ be the Haar computable copy of~$\+ W(G)$ obtained in \cref{ex:Qp}, and let $\wt {\+ V}$ be a further  Haar computable copy of~$\+ W(G)$. Using the notation of \cref{ex:Qp}, let 
 \bc $E_{z,r,a}= g^zD_{r,a}$ for each $z, r \in \ZZ, a \in C_{p^\infty}$.   \ec We list some properties of these elements of $\+ V$ that will  be needed shortly. Note that we can view $\+ W$ as embedded into $\+ V$ by identifying $\la r,a\ra$ with $\la 0,r,a\ra$. 
Also note   that $E_{z,r,a} \colon U_{r-z} \to U_r$ (using the category notation discussed after \cref{fact:WGGW}). Since $D_{r+1, a} \sub D_{r,pa}$, we have 
   \begin{equation} \label{eqn:EEE} E_{z, r+1, a } \sub E_{z,r,pa}. \end{equation} 
  Furthermore,   \begin{equation} \label{eqn:EEE2} E_{z, r, 0 }= g^{z} U_r= U_{r+z} g^{z}= (g^{-z} U_{r-z})^{-1}= (E_{-z, r-z, 0 })^{-1}. \end{equation}   
    
 By hypothesis on $\wt {\+ V}$, there is a  meet groupoid  isomorphism $\ol \Gamma \colon \+ V \to \wt {\+ V}$.     Since $G$ has no compact open subgroups besides the ones present in $\+ W(\QQ_p)$,   the  family  $(\wt U_r)_{r \in \ZZ}$, where  $\wt U_r = \ol \Gamma(U_r)$,      is   computable in $\wt {\+ V}$  by the same argument as before.   The set of elements  $A$  of $\wt {\+ V}$ that are a  left and a right coset of the same  subgroup is computable by checking whether $A^{-1} \cdot A= A \cdot A^{-1}$. The operations of $\wt {\+ V}$   induce a Haar computable meet groupoid   $\wt {\+ W}$ on this set.  Clearly  the restricted  map    $\Gamma =\ol \Gamma \mid  \+ W$  is an isomorphism $\+ W \to \wt {\+ W}$. So by the case of $\QQ_p$, there is a computable isomorphism $\Delta \colon  \+ W \to \wt {\+ W}$.  
% \sasha{OK now I think I see what you are doing, it makes sense now. I (as usual) wonder whether $\mathbb{Q}_p$ can be replaced with other c.c.~group, maybe with some extra conditions, for instance. I think there could be some much more general facts implicit to these results, but I cannot be bothered, and I suspect you neither:) We should leave the existence of such general sufficient condition as an open problem.  }
%  
 
  We will extend $\Delta$ to a computable  isomorphism $\ol \Delta \colon \+ V \to \wt {\+ V}$.   The following summarizes the setting: 
\[\xymatrix{
   \+ V   \ar[r]^{\ol \Gamma , \ol \Delta}& \wt{\+ V}  \\
    \+ W \ar[u]^{\sub}      \ar[r]_{\Gamma,\Delta}                     & \wt { \+ W} \ar[u]_{\sub} } \]%
In five phases we define a computable family $\wt E_{z,r,a}$ ($z, r \in \ZZ, a \in C_{p^\infty}$), and then let $\ol \Delta(E_{z,r,a})= \wt E_{z,r,a}$.     As before write $\wt D_{r,a} = \Delta( D_{r,a})$.

  \bi \item[(a)] Let $\wt E_{0,r,a} = \wt D_{r,a}$.
 Choose   $F_0:= \wt E_{-1,0,0} \colon \wt U_1 \to \wt U_0$
  \item[(b)] compute $F_r:=\wt E_{-1,r,0}\colon U_{r+1} \to U_r$ by recursion on $|r|$, where $r \in \ZZ$, in such a way  that $\wt F_{r+1} \sub \wt F_r$ for each $r \in \ZZ$; this is possible by (\ref{eqn:EEE}) and since $\+ V \cong \wt {\+ V}$ via $\ol \Gamma$. 
  \item[(c)]
   For $z < -1$, compute  $ \wt E_{z,r,0}\colon U_{r-z} \to U_r$ as follows: 
   \bc $\wt E_{z,r,0}= F_{r-z-1} \cdot   F_{r-z-2} \cdot \ldots \cdot   F_r$. \ec
  \item[(d)] For $z>0$ let $E_{z,r,0}= (\wt E_{-z, r-z, 0 })^{-1}$; this is enforced by (\ref{eqn:EEE2}). 
  \item[(e)] Let $\wt E_{z,r,a}= \wt E_{z,r,0} \cdot \wt D_{r,a}$. 
  \ei   
One verifies that $\ol \Delta$  preserves the   meet groupoid operations (we omit the formal detail).  To show that $\ol \Delta$ is onto, suppose that $\wt E \in \+ V$ is given. Then $\wt E = \Gamma(E_{z,r,a})$ for some $z,r,a$. By (\ref{eqn:EEE}) we may assume that $z<0$.  Then   $  E_{z,r,0}=   \prod_{i=1}^{-z}E_{-1,r-z-i,0} $ as above. So, writing $F_s$ for $\wt E_{-1,s,0}$, we have  $\Gamma(E_{z,r,0})= \prod_{i=1}^{-z}F_{r-z-i} \wt D_{r-z-i, a_i}$ for some $a_i \in C_{p^\infty}$. 

Note  that  $\wt S(D)= F \cdot D \cdot F^{-1}$ for each $D \in L(\wt U_r)  \cap \wt {\+ W}$ and $F \colon \wt U_{r+1} \to  \wt U_r$. 
For, the analogous statement clearly holds in $\+ V$;  then one uses  that  $\wt S=\Gamma \circ S \circ \Gamma^{-1}  $, and that $\ol \Gamma \colon \+ V \to \wt {\+ V}$ is an isomorphism.   
Since $\wt D_{r+1,a}= \wt S(\wt D_{r,a})$, we may conclude that  $\wt D_{r+1, a}\cdot F = F \cdot D_{r,a}$ for each such $F$. We can use these ``quasi-commutation relations"  to simplify the expression $\prod_{i=1}^{-z}F_{r-z-i} \wt D_{r-z-i, a_i}$ to  $\wt E_{z,r,0} \wt D_{r,b}$ for some $b \in C_{p^\infty}$. Hence $\wt E = \wt E_{z,r,0} \wt D_{r,b} \wt D_{r,a}$. This  shows that $\wt E$ is in the range of $\ol \Delta$, as required.
\end{proof}

\def\cprime{$'$} \def\cprime{$'$}

\
%% 
%\bibliographystyle{plain}
%
%\bibliography{../bibs/Nies,../bibs/randomness,../bibs/settheory,../bibs/various,../bibs/recursiontheory,../bibs/analysis,../bibs/Kucera,../bibs/modeltheory,../bibs/reverse_maths,../bibs/groups,../bibs/ergodic_theory,../bibs/computer_science,../bibs/quantum,../bibs/algebra}

\end{document}